\documentclass[12pt]{amsart}
\usepackage{anysize}

\usepackage{amsmath}
\usepackage{amssymb}
\usepackage[english]{babel}
\usepackage{enumerate}
\usepackage[naturalnames=true]{hyperref}
\usepackage{todonotes}
\usepackage{diagbox}
\usepackage{tikz}
\usepackage{ifthen}

\newboolean{buntig}
\setboolean{buntig}{false}

\newcommand{\gauss}[2]{\genfrac{[}{]}{0pt}{}{#1}{#2}}

\newcommand{\scrR}{\mathcal{R}}
\newcommand{\scrP}{\mathcal{P}}
\newcommand{\bbC}{\mathbb{C}}
\newcommand{\bbR}{\mathbb{R}}
\newcommand{\bbZ}{\mathbb{Z}}
\newcommand{\bbF}{\mathbb{F}}

\renewcommand{\epsilon}{e}
\newcommand{\even}{0}
\newcommand{\odd}{1}
\newcommand{\codim}{\text{codim}}

\newtheorem{satz}[equation]{Theorem}
\newtheorem{lemma}[equation]{Lemma}
\newtheorem{prop}[equation]{Proposition}
\newtheorem{kor}[equation]{Corollary}
\newtheorem{conjecture}[equation]{Conjecture}

\newtheorem{defi}[equation]{Definition}
\newtheorem{beispiel}[equation]{Example}

\newtheorem{Remark}[equation]{Remark}
\numberwithin{equation}{section}

\begin{document}

\title{Large $\{0, 1, \ldots, t\}$-Cliques in Dual Polar Graphs}

\author{Ferdinand Ihringer}
\address{ %
Justus-Liebig-Universit\"at,
Mathematisches Institut,
Arndtstra\ss e 2, D-35392 Gie\ss en,
Germany}
\email{Ferdinand.Ihringer@math.uni-giessen.de}

\author{Klaus Metsch}
\address{Justus-Liebig-Universit\"at,
Mathematisches Institut,
Arndtstra\ss e 2,
D-35392 Gie\ss en,
Germany.}
\email{Klaus.Metsch@math.uni-giessen.de}

\subjclass[2010]{51E20; 05B25; 52C10}

\begin{abstract}
  We investigate $\{0, 1, \ldots, t \}$-cliques of generators on dual polar graphs of finite classical polar spaces of rank $d$. These cliques are also known as Erd\H{o}s-Ko-Rado sets in polar spaces of generators with pairwise intersections in at most codimension $t$. Our main result is that we classify all such cliques of maximum size for $t \leq \sqrt{8d/5}-2$ if $q \geq 3$, and $t \leq \sqrt{8d/9}-2$ if $q = 2$.  We have the following byproducts.
  \begin{enumerate}[(a)]
   \item For $q \geq 3$ we provide estimates of Hoffman's bound on these $\{0, 1, \ldots, t \}$-cliques for all $t$.
   \item For $q \geq 3$ we determine the largest, second largest, and smallest eigenvalue of the
  graphs which have the generators of a polar space as vertices and where two generators are adjacent
  if and only if they meet in codimension at least $t+1$. Furthermore, we provide nice explicit formulas
  for all eigenvalues of these graphs.
    \item We provide upper bounds on the size of the second largest maximal $\{0, 1, \ldots, t \}$-cli\-ques for some $t$.
  \end{enumerate}
  
  \smallskip
\noindent \textbf{Erd\H{o}s-Ko-Rado Theorem; Polar Space; Distance-regular graph; Independent Set}
\end{abstract}

\maketitle

\section{Introduction}

A \textit{clique} of a graph is a set of pairwise adjacent vertices of a graph.
Determining the maximum size of a clique, the so-called clique number, is a classical problem in graph theory.
For some graphs, cliques are traditionally called Erd\H{o}s-Ko-Rado sets (\textit{EKR set}).
EKR sets were introduced by Erd\H{o}s, Ko, and Rado \cite{MR0140419} in 1961 as a family $Y$ of $k$-element subsets of $\{ 1, \ldots, n\}$ where the elements of $Y$ pairwise intersect in at least $t$ elements.
Erd\H{o}s, Ko, and Rado classified all such $Y$ of maximum size for $t=1$.

\begin{satz}[Theorem of Erd\H{o}s, Ko, and Rado]\label{thm_ekr}
  Let $n \geq 2k$. Let $Y$ be an EKR set of $k$-element subsets of $\{1, \ldots, n\}$. Then
  \begin{align*}
   |Y| \leq \binom{n-1}{k-1}
  \end{align*}
  with equality for $n>2k$ if and only if $Y$ is set of all $k$-sets containing a fixed element.
\end{satz}

For general $t$, the theorem looks as follows.

\begin{satz}\label{thm_ekr2}
  Let $n \geq 2k$. Let $t \in \{ 0, \ldots, k\}$.
  Let $Y$ be a set of $k$-element subsets of $\{1, \ldots, n\}$ such that $|K\cap K'|\ge t$ for all $K,K'\in Y$. Then
  \begin{align*}
   |Y| \leq \max \left\{ \binom{n-t}{k-t}, \binom{2k - t}{k} \right\}.
  \end{align*}
\end{satz}

These tight upper bounds for all $t$ on EKR sets of sets were given by Wilson in 1984 \cite{wilson_ekr_1984}.
The classification of all examples of maximum size was completed by Ahlswede and Khachatrian in 1997 \cite{ahlswede_ekr_1997}.
Many generalizations of the EKR problem exist for general $t$.
For example for vector spaces \cite{MR867648,MR0382015,MR2231096} and permutation groups \cite{MR2423345}.

With one exception, no attempts were made until now to investigate EKR sets of finite classical polar spaces in the general case.
This exception is the investigation of $\{0,1,2\}$-cliques of dual polar graphs by Brouwer and Hemmeter \cite{MR1158800}, where they classified all $\{0,1,2\}$-cliques on dual polar graphs in the non-Hermitian cases.
This problem was modified by De Boeck \cite{maarten_ekr_planes} to EKR sets, where he classified EKR sets $Y$ of planes (not necessarily generators) for $d \in \{3, 4, 5\}$ and $|Y| \geq 3q^4 + 3q^3 + 2q^2 + q + 1$. Here $q$ is the {\it order} of the polar space, that is the order of its underlying field.
For the more restricted problem in sense of Theorem \ref{thm_ekr} Stanton proved upper bounds in \cite{MR578319}.
The largest examples were mostly classified by Pepe, Storme, and Vanhove in \cite{MR2755082}. For the remaining open case see \cite{Ihringer_Metsch,Metsch}

Define a \textit{$(d, t)$-EKR set} of generators of a polar space of rank $d$ to be a set $Y$ of generators of the polar space such that $y_1, y_2 \in Y$ implies $\codim(y_1 \cap y_2) := d - \dim(y_1 \cap y_2) \leq t$. In this notation Brouwer and Hemmeter investigated $(d, 2)$-EKR sets of finite classical polar spaces.
This paper is concerned about generalizing their work to $(d, t)$-EKR sets ($\{0, 1,\ldots, t\}$-cliques) of maximum size for more values of $t$.
We provide sharp upper bounds for $(d, t)$-EKR sets for $t \leq \sqrt{8d/5}-2$ if $q \geq 3$ and for $t \leq \sqrt{8d/9}-2$ if $q=2$ (Theorem \ref{thm_max_t_even}, Theorem \ref{thm_max_t_odd}, and Theorem \ref{thm_apprx}).
These results imply upper bounds on the size of the second largest example, so they might provide a reasonable basis
to classify the second largest maximal $(d, t)$-EKR sets as it was done for EKR sets of sets \cite{MR0219428}, vector spaces \cite{MR2651724}, and
some special cases in polar spaces \cite{maarten_ekr_hyperbolic, maarten_ekr_planes}.
Furthermore, we give non-trivial upper bounds for general $t$, $q \geq 3$ (Theorem \ref{kor_ev_upper_bnd}).
As a side effect we determine the smallest, largest, and second largest eigenvalues of the adjacency matrix of the considered associated graph for $q \geq 3$ in
Theorem \ref{thm_formula_ev}. These numbers alone are important parameters of a graph
as they can be used easily to make non-trivial statements on many other properties of the graph such as the chromatic number or the convergence of random walks.
Also noteworthy are the inequalities on the Gaussian coefficients and the number of generators of a polar space given in Section \ref{sec_ests}
which are more accurate than the usual approximations, but still so simple that they can be used easily for other combinatorial problems in
polar spaces and vector spaces.

Our main result is the following.

\begin{satz}\label{thm_apprx}
Let $q$ be a prime power and let $t$ and $d$ be non-negative integers satisfying
\begin{enumerate}[(i)]
  \item $0\le t \leq \sqrt{\frac{8d}{9}} - 2$, if $q=2$,
  \item $0\le t \leq \sqrt{\frac{8d}{5}} - 2$, if $q \geq 3$.
\end{enumerate}
Then all finite classical polar spaces of rank $d$ and order $q$ satisfy the following.
\begin{enumerate}[(a)]
    \item If $t$ is even, then the largest $(d, t)$-EKR set is (up to isomorphism) the set of all generators that meet a fixed $d$-space in a subspace of dimension at least $d-\frac{t}{2}$
    \item If $t$ is odd, then the largest $(d, t)$-EKR set is (up to isomorphism) the set of all generators that meet a fixed $(d-1)$-space in a subspace of dimension at least $d-\frac{t}{2} - \frac{1}{2}$.
    \end{enumerate}
\end{satz}

Beyond these results we hope that the used technique, which combines algebraic and geometrical arguments, is applicable to EKR problems in other interesting structures, and can be modified to classify all $(d, t)$-EKR sets of generators of maximum size for more values of $d$ and $t$.

This paper is organized as follows.
The main parts are Section \ref{sec_t_even} and Section \ref{sec_t_odd}, where we develop stability results for $(d, t)$-EKR sets (Theorem \ref{thm_max_t_even} and Theorem \ref{thm_max_t_odd}) which depend on the maximum size of $(2t-1, t)$- and $(2t-2, t)$-EKR sets.
In Section \ref{sec_ests} we calculate some inequalities on Gaussian coefficients and the number of generators of polar spaces.
We use these to approximate Hoffman's bound for $(d, t)$-EKR sets in Section \ref{sec_hoffman} for $q \geq 3$.
Finally, in Section \ref{sec_felina} we prove Theorem \ref{thm_apprx}.
The other sections are devoted to definitions and tedious, but necessary calculations.

{\bf Remark}. In order to increase the readability of the paper, we omit the proof for the case when $q=2$, since this case requires different estimations. A proof can be found in the Ph.D. thesis of the first author \cite{ihringer_phd}.

\section{Combinatorial Properties of Polar Spaces}\label{sec_def}

Finite classical polar spaces consist of the totally isotropic or totally singular subspaces of a non-degenerate sesquilinear, respectively, quadratic form on a vector space partially ordered by inclusion. For details we refer to standard references such as \cite{hirschfeld1991general}. The maximal totally isotropic (or totally singular) subspaces of a finite classical polar space are called its \textit{generators}. If the (vector space) dimension of a generator is larger than $3$, then all finite polar spaces are \textit{classical}. The dimension of all generators of a polar space is the same and this dimension is called the \textit{rank} of a polar space.

For any prime power $q$, there exist the following types of polar spaces of rank $d$ and order $q$:
\begin{enumerate}[(a)]
 \item The hyperbolic quadric $Q^+(2d-1, q)$.
  Up to coordinate transformation it is defined by the quadratic form $f(x)=x_0x_1 + \ldots + x_{2d-2}x_{2d-1}$.
 \item The parabolic quadric $Q(2d, q)$.
  Up to coordinate transformation it is defined by the quadratic form $f(x)=x_0^2 + x_1x_2 + \ldots + x_{2d-1}x_{2d}$.
 \item The elliptic quadric $Q^-(2d+1, q)$.
  Up to coordinate transformation it is defined by the quadratic form $f(x)=h(x_0, x_1) + x_2x_3 + \ldots + x_{2d}x_{2d+1}$, where $h(x_0, x_1)$ is an irreducible homogenous quadratic polynomial over $\bbF_q$.
 \item The Hermitian polar space $H(2d-1, q)$ when $q=r^2$ is a square.
  Its standard sesquilinear form is $f(x,y)=x_0y_0^r+\ldots+x_{2d-1}y_{2d-1}^r$.
 \item The Hermitian polar space $H(2d, q)$ when $q=r^2$ is a square.
  Its standard sesquilinear form is $f(x,y)=x_0y_0^r+\ldots+x_{2d}y_{2d}^r$.
 \item The symplectic polar space $W(2d-1, q)$.
  Its standard sesquilinear form is $f(x,y) = x_0y_1 - x_1y_0 + \ldots + x_{2d-2} y_{2d-1} - x_{2d-1} y_{2d-2}$.
\end{enumerate}

Set $\epsilon = 0, \frac12, 1, 1, \frac32, 2$ for $Q^+(2d-1, q)$, $H(2d-1, q)$, $Q(2d, q)$, $W(2d-1, q)$, $H(2d, q)$, respectively, $Q^-(2d+1, q)$.
A polar space $\scrP$ of rank $d$, order $q$, and type $e$ has exactly
\begin{align}
  \prod_{i=0}^{d-1} (q^{i+\epsilon} + 1)\label{number_gens}
\end{align}
generators, see Appendix VI in \cite{hirschfeld1991general}.

\begin{Remark}
  We shall use the following conventions.
  \begin{enumerate}[(a)]
   \item Unless otherwise mentioned, we are always using vector space dimension and never projective dimension.
    This increases the readability of all the used eigenvalue formulas.
   \item Whenever we say \textit{totally isotropic}, then we mean \textit{totally singular} if the considered polar space is
    a quadric.
   \item The parameters $\epsilon$ and $q$ are always fixed.
   \item Sometimes we write $k$-space, $k$-subspace, $k$-dimensional subspace for a subspace of dimension $k$.
  \end{enumerate}
\end{Remark}

For integers $n$ and $k$ define the \emph{Gaussian coefficient} $\gauss{n}{k}_q$ by
\begin{align*}
  \gauss{n}{k}_q = \begin{cases}
		   \prod_{i=1}^{k} \frac{q^{n-i+1}-1}{q^i-1} & \text{ if } 0 \leq k \leq n,\\
                    0 & \text{ otherwise.}
                  \end{cases}
\end{align*}
We write $\gauss{n}{k}$ instead of $\gauss{n}{k}_q$ when $q$ is clear from the context.
It is well-known that the number of $k$-dimensional subspaces of a vector space of dimension $n$ equals $\gauss{n}{k}$.
In particular, an $n$-dimensional vector space has exactly
\begin{align}
  \gauss{n}{n-k} = \gauss{n}{k} \label{gauss_duality}
\end{align}
subspaces of dimension $k$. A straightforward calculation shows
\begin{align}
 \gauss{n+1}{k+1} = q^{n-k} \gauss{n}{k} + \gauss{n}{k+1}\label{gauss_rec}
\end{align}
for integers $k$ and $n \geq 0$.

\begin{lemma}\label{gauss_alt_sum}\label{identity_0}
  Let $n \geq 1$ and $a \geq 0$. Then,
  \begin{align*}
   \sum_{k=0}^{a} (-1)^{n-k} \gauss{n}{k} q^{\binom{n-k}{2}} = (-1)^{n+a} \gauss{n-1}{a} q^{\binom{n-a}{2}}.
  \end{align*}
\end{lemma}
\begin{proof}
  \begin{align*}
   &\sum_{k=0}^{a} (-1)^{n-k} q^{\binom{n-k}{2}} \gauss{n}{k} \\
   \stackrel{\eqref{gauss_rec}}{=} &\sum_{k=0}^{a} (-1)^{n-k} q^{\binom{n-k}{2}} \left( q^{n-k} \gauss{n-1}{k-1} + \gauss{n-1}{k} \right)\\
   = &\sum_{k=0}^{a} (-1)^{n-k} \left(  q^{\binom{n-k+1}{2}} \gauss{n-1}{k-1} +  q^{\binom{n-k}{2}} \gauss{n-1}{k} \right)\\
   = &(-1)^{n-a} q^{\binom{n-a}{2}} \gauss{n-1}{a}.
  \end{align*}
\end{proof}

\begin{defi}
  For integers $d,r, s, u, z, z_1, z_2, t$ with $d\ge 0$ let
  \begin{enumerate}[(a)]
   \item $\psi_{12}(d, r, s, u)$ be the number of $r$-spaces in $\bbF_q^d$ meeting a fixed $s$-space in a fixed $u$-space, if $0 \leq u \leq s \leq d$, and $0$ otherwise.
   \item $\psi_{2}(d, r, s, u)$ be the number of $r$-spaces in $\bbF_q^d$ meeting a fixed $s$-space in some $u$-space, if $0 \leq s \leq d$, and $0$ otherwise.
   \item $\psi_3(d, x, y, z, z_1, z_2)$ be the number of $z$-spaces in $\bbF_q^d$ that meet a fixed $x$-space $X$ in some $z_2$-space and a fixed $y$-space $Y \subseteq X$ in some $z_1$-space, if $0 \leq y \leq x \leq d$, and $0$ otherwise.
   \item
   \begin{align*}
    \psi^\even=\sum_{i=1}^{t/2-1} \psi_3(d, d - \frac{3}{2} t, d - 2t + 1, d - \frac{t}{2}, d - \frac{5}{2}t + 1+i, d-2t)
   \end{align*}
   when $t$ is even.
   \item
   \begin{align*}
      \psi^\odd=\sum_{i=1}^{t/2-3/2} \psi_3(d-1, d - \frac{3}{2}t + \frac{1}{2}, d-2 t + 2, d-\frac{1}{2}t - \frac{1}{2}, d - \frac{5}{2} t+ \frac{5}{2} +i, d-2t+1)
   \end{align*}
   when $t$ is odd.
   \item $\overline{\psi}^\odd$ be $\psi_2(d, d-\frac{t}{2}+\frac{1}{2}, d - \frac{3}{2}t+\frac{1}{2}, d-2t+1)$ for $t$ odd.
   \item $\omega(d, r)$ be the number of generators that contain a fixed $(d-r)$-space in a polar space of rank $d$ if
   $0 \leq r \leq d$, and $0$ otherwise.
   \item $c_{d, t}$ be the maximum size of a $(d, t)$-EKR set of generators of a finite classical polar space of rank $d$.
  \end{enumerate}
\end{defi}

\begin{lemma}~\label{lem_basic_cnt_pgnq}\label{kor_basic_cnt_app1}\label{lem_upperbound_arbP_t_even_spec}\label{kor_upperbound_arbP_t_even}\label{kor_upperbound_arbP_t_odd}
  \begin{enumerate}[\rm (a)]
  \item
    \begin{align*}
      \psi^\even = q^{\frac{3}{4} t^2} \sum_{i=1}^{t/2-1} q^{(\frac{t}{2}-1-i)(\frac{t}{2}-i)} \gauss{d-2t+1}{t/2 - i} \gauss{t/2-1}{i},
    \end{align*}
  \item
    \begin{align*}
      \psi^\odd =  q^{(\frac{t}{2}-\frac{1}{2})(\frac{3}{2}t-\frac{3}{2})} \sum_{i=1}^{t/2-3/2} q^{(\frac{t}{2}-\frac{3}{2}-i)(\frac{t}{2}-\frac{1}{2}-i)}
      \gauss{d-2t+2}{(t-1)/2-i} \gauss{\frac{t}{2} - \frac{3}{2}}{i},
    \end{align*}
  \item
    \begin{align*}
      \overline{\psi}^\odd = q^{(\frac{3}{2}t - \frac{1}{2})(\frac{t}{2} - \frac{1}{2})} \gauss{d - \frac{3}{2}t + \frac{1}{2}}{(t-1)/2},
    \end{align*}
    \item
    \begin{align*}
      \omega(d, r) = \prod_{i=0}^{r-1} (q^{i + \epsilon} + 1).
    \end{align*}
  \end{enumerate}
\end{lemma}
\begin{proof}
  First we calculate $\psi_3(d, x, y, z, z_1, z_2)$ for $0\le y\le x\le d$.
  For this consider an $x$-dimensional subspace $X$ of $\bbF_q^d$ and a $y$-dimensional subspace $Y$ with $Y \subseteq X$.
  We want to choose a $z$-dimensional subspace $Z$, where
   \begin{align*}
    \dim(Z \cap Y) = z_1 \text{ and }
    \dim(Z \cap X) = z_2.
   \end{align*}
    The number of ways choosing $Z \cap X$ is $\psi_{2}(x, z_2, y, z_1)$.
    Then the number of ways choosing $Z$ through a fixed subspace $Z \cap X$ is $\psi_{12}(d, z, x, z_2)$.
    Hence, $\psi_3(d, x, y, z, z_1, z_2) = \psi_{2}(x, z_2, y, z_1) \psi_{12}(d, z, x, z_2)$.
    By \cite[Th. 3.3, (1), p. 88]{hirschfeld1998projective},
  \begin{align*}
    \psi_{12}(d, r, s, u) = q^{(r-u)(s-u)} \gauss{d-s}{r-u}.
  \end{align*}
  By \cite[Th. 3.3, (2), p. 88]{hirschfeld1998projective},
  \begin{align*}
     \psi_{2}(d, r, s, u) = \gauss{s}{u}\psi_{12}(d, r, s, u).
  \end{align*}
  Hence, $\psi_3(d, x, y, z, z_1, z_2) = $
 \begin{align*}
    q^{(z_2-z_1)(y-z_1)+(z-z_2)(x-z_2)} \gauss{y}{z_1}\gauss{x-y}{z_2-z_1} \gauss{d-x}{z-z_2}.
 \end{align*}
    These equations and \eqref{gauss_duality} imply the first three assertions.

   The number $\omega(d, r)$ equals the number of generators in the quotient geometry of a $(d-r)$-space.
      That is a polar space of the same type with generators of rank $r$.
      The claim follows from \eqref{number_gens}.
\end{proof}

\section{A Property of $(d,t)$-EKR Sets}\label{sec_basics}

A $(d, t)$-EKR set is \emph{maximal} if it is not a proper subset of another $(d, t)$-EKR set.
We need the following basic result on maximal $(d, t)$-EKR sets.

\begin{lemma}\label{lem_max_dmtp1_is_not_max_dmt_ekr_set}
For $0\le t\le d$ a $(d,t-1)$-EKR set of a polar space is never a maximal $(d,t)$-EKR set of that polar space.
\end{lemma}
\begin{proof}
  If $t = 0$, then the unique maximal $(d, -1)$-EKR set is the empty set, but every maximal $(d, 0)$-EKR set is the set of one generator.

  Suppose now that $Y$ is a $(d, t-1)$-EKR set with $t>0$.

  \textbf{Case 1.} If there are no $a_1, a_2 \in Y$ such that $\dim(a_1 \cap a_2) = d-t+1$, then
  (by induction on $t$) $Y$ is not a maximal $(d, t-1)$-EKR set, hence also not
  a maximal $(d,t)$-EKR set.

  \textbf{Case 2.} There exist $a_1, a_2 \in Y$ such that $\dim(a_1 \cap a_2) = d-t+1$.
  Take a $(d-1)$-dimensional subspace $a_1'$ of $a_1$ such that $\dim(a_1' \cap a_2) = d-t$. There exists a generator $c$ through $a'_1$ with $c\cap a_2=a'_1\cap a_2$. Then $c\notin Y$.  Since all $b \in Y$ satisfy $\dim(c\cap b)\ge \dim(a_1' \cap b) \geq \dim(a_1 \cap b)-1 = d-t$, $Y\cup \{c\}$ is still a $(d, t)$-EKR set.
\end{proof}

\begin{lemma}\label{lem_d_t_eq_d_1}
  The largest $(d, 1)$-EKR sets of generators consist of all generators on a subspace of dimension $d-1$.
\end{lemma}
\begin{proof}
Let $Y$ be a $(d,1)$-EKR set of generators. We may suppose $|Y| > 1$ and consider distinct $a, b \in Y$. Then $a \cap b$ is a subspace of dimension $d-1$. We want to show that all elements of $Y$ contain $a \cap b$. Assume on the contrary that there exists a $c \in Y$ not containing $a \cap b$. As $c$ meets $a$ and $b$ in a subspace of dimension $d-1$, there exist points $p\in (a\cap c)\setminus a\cap b$ and $p'\in (b\cap c)\setminus a\cap b$. But then $\langle a \cap b,p,p'\rangle$ is a totally isotropic subspace containing the generators $a$ and $b$, contradiction.
\end{proof}

\section{EKR Sets, $t$ even}\label{sec_t_even}

Throughout this section we work in a finite classical polar space of rank $d > 2$ and given type $\epsilon$. We assume throughout this section that $t$ is an even integer satisfying $d\ge 2t\ge 0$.

\begin{defi} We define the constants $b_1^\even$ and $b_2^\even$ by
 \begin{align*}
    &b_1^\even = \gauss{d - \frac{3}{2} t}{t/2-1} c_{ 2t-1, t}\\
    &b_2^\even = q^{\epsilon \frac{t}{2} + \binom{t/2}{2}} \psi^\even .
 \end{align*}
\end{defi}

\begin{lemma}\label{lem_bounds_t_even}
  Let $Y$ be a $(d, t)$-EKR set.
 \begin{enumerate}[(a)]
  \item Let $P$ be a subspace of dimension at least $d-\frac{3}{2}t$. If $\dim(c \cap P) > \dim(P) - \frac{t}{2}$ for all elements $c$ of $Y$, then $Y$ has at most $b_1^\even$ elements.
  \item Let $U$ be a generator, $P$ a subspace of $U$ of dimension $d-\frac{3}{2}t$, and $A$ a subspace of $P$ with $\dim(A) \geq d - 2t+1$. If all elements $c$ of $Y$ satisfy $\dim(U \cap c) = d-\frac{t}{2}$, $\dim(c \cap P) = d-2t$, and $\dim(c \cap A) \geq \dim(A) - \frac{t}{2} + 1$, then $Y$ has at most $b_2^\even$ elements.
 \end{enumerate}
\end{lemma}
\begin{proof}
  \begin{enumerate}[(a)]
  \item By replacing $P$ if necessary by a subspace of $P$ of dimension $d-\frac{3}{2}t$, we may assume that $\dim(P) = d-\frac{3}{2}t$. Then the Gaussian coefficient in the definition of $b_1^\even$ is the number of subspaces $U$ of $P$ of codimension $\frac{t}{2} - 1$. By hypothesis, every element of $Y$ contains one such subspace $U$. The elements of $Y$ on such a fixed subspace $U$ form a $(2t-1, t)$-EKR set in the quotient geometry of $U$ and hence there are at most $c_{2t-1,t}$ such elements.
  \item By replacing $A$ if necessary by a subspace of $A$ of dimension $d-2t+1$, we may assume that $\dim(A) = d-2t+1$.
  There are $\psi^\even$ subspaces $T$ of $U$ with $\dim(T) = d - \frac{t}{2}$, $\dim(T \cap P) = d - 2 t$,
  and $\dim(T \cap A) = \dim(A) - \frac{t}{2}  + i$ with $i \in \{ 1, \ldots, t/2-1\}$.
  For each such $T$ consider the quotient geometry $T^\perp/T$ which is isomorphic
  to a polar space of the same type with rank $\frac{t}{2}$.
  It is well-known (see for example Corollary \ref{kor_dis_gens}) that there are exactly $q^{\epsilon \frac{t}{2} + \binom{t/2}{2}}$ generators in $T^\perp/T$
  disjoint to $U/T$. Hence, there are exactly $q^{\epsilon \frac{t}{2} + \binom{t/2}{2}}$ generators $a$ with $a \cap U = T$.
 \end{enumerate}
\end{proof}

\ifbuntig
  \definecolor{circle1}{RGB}{66,91,151}
  \definecolor{circle2}{RGB}{132,172,102}
  \definecolor{circle3}{RGB}{212,103,103}
\else
  \definecolor{circle1}{RGB}{210,210,210}
  \definecolor{circle2}{RGB}{210,210,210}
  \definecolor{circle3}{RGB}{210,210,210}
\fi

\pgfdeclarelayer{circle layer}
\pgfdeclarelayer{text layer}
\pgfsetlayers{main,circle layer,text layer}

\begin{figure}[ht]
 \centering
\begin{tikzpicture}[y=0.80pt, x=0.8pt,yscale=-1, inner sep=0pt, outer sep=0pt]
  \begin{scope}[scale=0.24]
  \begin{pgfonlayer}{circle layer}
\ifbuntig
  \path[shift={(-65.71429,1.42857)},draw=black,miter limit=4.00,line
    width=2.000pt] (654.2857,518.0765)arc(0.000:180.000:227.142855 and
    238.571)arc(-180.000:0.000:227.142855 and 238.571) -- cycle;
  \path[cm={{1.78788,0.0,0.0,1.20565,(-429.39395,-136.79339)}},draw=black,fill=circle1,opacity=0.500,miter
    limit=4.00,line width=1.000pt] (568.5714,689.5051)arc(0.000:180.000:117.857140
    and 177.143)arc(-180.000:0.000:117.857140 and 177.143) -- cycle;
  \path[cm={{1.78788,0.0,0.0,1.20565,(-601.53681,-388.93629)}},draw=black,fill=circle2,opacity=0.500,miter
    limit=4.00,line width=1.000pt] (568.5714,689.5051)arc(0.000:180.000:117.857140
    and 177.143)arc(-180.000:0.000:117.857140 and 177.143) -- cycle;
  \path[cm={{1.78788,0.0,0.0,1.20565,(-310.10824,-413.22198)}},draw=black,fill=circle3,opacity=0.500,miter
    limit=4.00,line width=1.000pt] (568.5714,689.5051)arc(0.000:180.000:117.857140
    and 177.143)arc(-180.000:0.000:117.857140 and 177.143) -- cycle;
\else
  \path[shift={(-65.71429,1.42857)},draw=black,miter limit=4.00,line
    width=2.000pt] (654.2857,518.0765)arc(0.000:180.000:227.142855 and
    238.571)arc(-180.000:0.000:227.142855 and 238.571) -- cycle;
  \path[cm={{1.78788,0.0,0.0,1.20565,(-429.39395,-136.79339)}},draw=black,fill=circle1,opacity=0.500,miter
    limit=4.00,line width=1.000pt] (568.5714,689.5051)arc(0.000:180.000:117.857140
    and 177.143)arc(-180.000:0.000:117.857140 and 177.143) -- cycle;
  \path[cm={{1.78788,0.0,0.0,1.20565,(-601.53681,-388.93629)}},draw=black,fill=circle2,opacity=0.500,miter
    limit=4.00,line width=1.000pt] (568.5714,689.5051)arc(0.000:180.000:117.857140
    and 177.143)arc(-180.000:0.000:117.857140 and 177.143) -- cycle;
  \path[cm={{1.78788,0.0,0.0,1.20565,(-310.10824,-413.22198)}},draw=black,fill=circle3,opacity=0.500,miter
    limit=4.00,line width=1.000pt] (568.5714,689.5051)arc(0.000:180.000:117.857140
    and 177.143)arc(-180.000:0.000:117.857140 and 177.143) -- cycle;
\fi
  \end{pgfonlayer}
  \end{scope}
  \begin{pgfonlayer}{text layer}
    \node at (-0.7cm, 5.2cm) {$U = \langle \ell_{12}, \ell_{13}, \ell_{23} \rangle$};
    \node at (0.6cm, 2.7cm) {$a_1$};
    \node at (4.1cm, 2.3cm) {$a_2$};
    \node at (2.6cm, 5.6cm) {$a_3$};
    \node at (2.4cm, 2.7cm) {$\ell_{12}$};
    \node at (3.1cm, 3.8cm) {$\ell_{23}$};
    \node at (1.85cm, 3.9cm) {$\ell_{13}$};
    \node at (2.425cm, 3.55cm) {$P$};
  \end{pgfonlayer}
\end{tikzpicture}
 \caption{The setting of Lemma \ref{lem_collected_properties_t_even}.}
 \label{fig:std_config}
\end{figure}
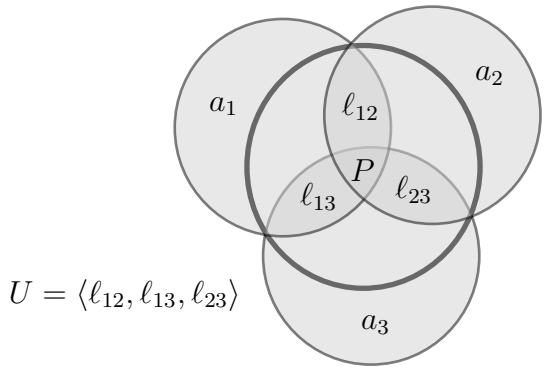

\begin{lemma}\label{lem_collected_properties_t_even}
 Let $Y$ be a $(d, t)$-EKR set, and consider $a_1, a_2, a_3 \in Y$. Then the following holds true.
 \begin{enumerate}[(a)]
  \item The dimension of $a_1 \cap a_2 \cap a_3$ is at least $d - \frac{3}{2} t$.
  \item Suppose that equality holds in Part (a) and put $U := U_{123}$ and $P := a_1 \cap a_2 \cap a_3$. Then
  \begin{enumerate}[(i)]
   \item $\dim(U) = d$.
   \item $\dim(\ell_{ij}) = d- t$ for $1 \leq i < j \leq 3$.
   \item $\dim(a_i \cap U) = d- \frac{t}{2}$ and $a_i \cap U = \langle \ell_{ij}, \ell_{ik} \rangle$ for $\{i, j, k\} = \{1, 2, 3\}$.
   \item Every $b \in Y$ satisfies $\dim(b \cap U) \geq d - \frac{t}{2}$ or $\dim(b \cap P) > d-2t$.
  \end{enumerate}
 \end{enumerate}
\end{lemma}
\begin{proof}
 \begin{enumerate}[(a)]
  \item As $Y$ is a $(d, t)$-EKR set, then $\dim(\ell_{ij}) \geq d-t$. As $U$ is totally isotropic, then $\dim(U) \leq d$. Since $\ell_{12} \cap \ell_{23} = P$ as well as $\langle \ell_{12}, \ell_{13} \rangle \cap \ell_{23} = P$ (because $\langle \ell_{12}, \ell_{13} \rangle \cap \ell_{23} \subseteq a_1 \cap \ell_{23} = P$), then
  \begin{align*}
   d &\geq \dim(U) = \dim( \langle \ell_{12}, \ell_{13}, \ell_{23}\rangle)\\
   &= \dim(\langle \ell_{12}, \ell_{13} \rangle) + \dim(\ell_{23}) - \dim(P)\\
   &= \dim(\ell_{12}) + \dim(\ell_{13}) - \dim(P) + \dim(\ell_{23}) - \dim(P)\\
   &\geq 3(d-t) - 2\dim(P).
  \end{align*}
  Hence, $\dim(P) \geq d- \frac{3}{2} t$.
  \item As $\dim(P) = d - \frac{3}{2} t$, the argument in (a) shows that $\dim(U) = d$ and $\dim(\ell_{ij}) = d-t$ for all $i, j$.
  For $\{ i,j,k\} = \{1, 2, 3\}$, we have $\ell_{ij}, \ell_{ik} \subseteq a_i \cap U$, and so $U = \langle a_i \cap U, \ell_{jk}\rangle$.
  Also $a_i \cap U \cap \ell_{jk} = a_i \cap \ell_{jk} = P$, and hence
  \begin{align*}
   d &= \dim(U) = \dim(\langle a_i \cap U, \ell_{jk}\rangle) = \dim(a_i \cap U) + \dim(\ell_{jk}) - \dim(P)\\
   &= \dim(a_i \cap U) + (d-t) - (d-\frac{3}{2} t).
  \end{align*}
  Therefore, $\dim(a_i \cap U) = d - \frac{t}{2}$ which implies that $a_i \cap U = \langle \ell_{ij}, \ell_{ik}\rangle$. We have proved the first three statements. For the final part, consider $b \in Y$.

  Part (a) shows that $\dim( b \cap \ell_{ij}) \geq d - \frac{3}{2} t$.
  Put $\ell_{ij}' := b \cap \ell_{ij}$.
  Then $\langle \ell_{12}', \ell_{13}'\rangle \cap \ell_{23}' \subseteq (a_1 \cap b) \cap \ell_{23} = b \cap P$ and hence equality holds. Clearly, $\ell_{12}' \cap \ell_{13}' = b \cap P$.
  This implies that
  \begin{align*}
   \dim(b \cap U) &\geq \dim(\langle \ell_{12}', \ell_{13}', \ell_{23}' \rangle)\\
   &\geq \dim(\ell_{12}') + \dim(\ell_{13}') + \dim(\ell_{23}') - 2 \dim(b \cap P)\\
   &\geq 3(d-\frac{3}{2}t) - 2\dim(b \cap P).
  \end{align*}
  Hence, if $\dim(b \cap P) \leq d-2t$, then $\dim(b \cap U) \geq d- \frac{t}{2}$.
 \end{enumerate}
\end{proof}

\begin{lemma}\label{lem_upperbound_1_t_even}
 If $\dim(a_1 \cap a_2 \cap a_3) > d- \frac{3}{2}t$ for all $a_1, a_2, a_3$ of a maximal $(d, t)$-EKR set $Y$, then $|Y| \leq b_1^\even$.
\end{lemma}
\begin{proof}
 By Lemma \ref{lem_max_dmtp1_is_not_max_dmt_ekr_set}, there exist $a_1, a_2 \in Y$ with $\dim(a_1 \cap a_2) = d-t$. Consider a third element $a_3 \in Y$ and put $P := a_1 \cap a_2 \cap a_3$. Consider any element $b \in Y$. By hypothesis, $\dim(b \cap a_1 \cap a_2) \geq d - \frac{3}{2}t +1$. As $P$ and $b \cap a_1 \cap a_2$ lie in $a_1 \cap a_2$, the dimension formula shows that
 \begin{align*}
  \dim(b \cap P) &= \dim((b \cap a_1 \cap a_2) \cap P)\\
  &\geq \dim(b \cap a_1 \cap a_2) + \dim(P) - \dim(a_1 \cap a_2)\\
  &\geq \dim(P) + 1 - \frac{t}{2}.
 \end{align*}
  Lemma \ref{lem_bounds_t_even} shows that $|y| \leq b_1^\even$.
\end{proof}

\begin{lemma}\label{lem_t_even_P_large2}
 Let $Y$ be a $(d, t)$-EKR set such that there exists a generator $U$ and a $(d-\frac{3}{2}t)$-space $P \subseteq U$ such that $a \in Y$ implies that
 \begin{align*}
  \dim(a \cap U) \geq d- \frac{t}{2} \text{ or } \dim(a \cap P) > d-2t.
 \end{align*}
  If $\dim(a \cap U) < d-\frac{t}{2}$ for at least one element $a$ of $Y$, then $|Y| \leq b_1^\even + b_2^\even$.
\end{lemma}
\begin{proof}
  Define
  \begin{align*}
   Y_1 &= \{ a \in Y: \dim(a \cap U) = d- \frac{t}{2}, \dim(a \cap P) = d-2t\},\\
   Y_2 &= \{ a \in Y: \dim(a \cap P) > d-2t\},
  \end{align*}
  We show first that $Y = Y_1 \cup Y_2$. For this, let $a\in Y$. Then
 \begin{align*}
 \dim(a\cap U)&=\dim(\langle a \cap U, P \rangle)+\dim(a\cap P)-\dim(P)
 \\
   &\leq \dim(U)+\dim(a\cap P)-\dim(P)
 \\
   &\leq \dim(a\cap P)+\frac32t.
 \end{align*}
 Hence, if $a\notin Y_2$, then the hypothesis of the lemma implies that $a\in Y_1$. Thus we have proved that $Y = Y_1 \cup Y_2$.

  The first part of Lemma \ref{lem_bounds_t_even} gives $|Y_2| \leq b_1^\even$.
  By hypothesis, there exists an element $a_2 \in Y$ with $\dim(a_2 \cap U) < d- \frac{t}{2}$.
  Then $a_2 \cap P$ has dimension at least $d-2t+1$.
  We shall show that $\dim(a_1 \cap a_2 \cap P) > \dim(a_2 \cap P) - \frac{t}{2}$ for all $a_1 \in Y_1$.
  Then the second part of Lemma \ref{lem_bounds_t_even} with $A := a_2 \cap P$ gives $|Y_1| \leq b_2^\even$ and we are done.
  Consider any element $a_1 \in Y_1$.

  We want to show that $\langle a_1 \cap P, a_2 \cap P \rangle$ is a proper subspace of $P$.
  Suppose to the contrary that $\langle a_1 \cap P, a_2 \cap P \rangle = P$. As $\dim(U) - \dim(a_1 \cap U) = t/2 = \dim(P) - \dim(a_1 \cap P)$, this implies that $U = \langle a_1 \cap U, a_2 \cap P\rangle$. Hence, every point of $a_1 \cap a_2$ lies in $U^\bot$; but $U$ is a generator, so $a_1 \cap a_2 \subseteq U$. It follows that
  \begin{align*}
   \dim(a_2 \cap U) &\geq \dim(\langle a_1 \cap a_2, a_2 \cap P\rangle)\\
   &= \dim(a_1 \cap a_2) + \dim(a_2 \cap P) - \dim(a_1 \cap a_2 \cap P)\\
   &= \dim(a_1 \cap a_2) - \dim(a_1 \cap P) + \dim(\langle a_1 \cap P, a_2 \cap P \rangle)\\
   &= \dim(a_1 \cap a_2) + \dim(P) - \dim(a_1 \cap P)\\
   &= \dim(a_1 \cap a_2) + \frac{t}{2} \geq d- \frac{t}{2}.
  \end{align*}
  Here we use $\dim(a_1\cap a_2) \geq d-t$, since $a_1, a_2 \in Y$.
  This contradicts $\dim(a_2 \cap U) < d- \frac{t}{2}$.

  Hence $\langle a_1 \cap P, a_2 \cap P\rangle$ is a proper subspace of $P$ and thus has dimension at most $\dim(P) - 1 = d-\frac{3}{2}t - 1$. It follows that
  \begin{align*}
   \dim(a_1 \cap a_2 \cap P) &= \dim((a_1 \cap P) \cap (a_2 \cap P))\\
   &\geq \dim(a_1 \cap P) + \dim(a_2 \cap P) - (d- \frac{3}{2}t - 1)\\
   &= \dim(a_2 \cap P) - \frac{t}{2} + 1.
  \end{align*}
  This completes the proof.
\end{proof}

\begin{beispiel}\label{lem_ex_t_even}
 The set consisting of all generators that meet a given generator in a subspace of dimension at least $d-\frac{t}{2}$ is a maximal $(d, t)$-EKR set.
\end{beispiel}
\begin{proof}
 As the given generator $U$ has dimension $d$, the dimension formula shows that the meet of two elements of $Y$ has dimension at least $d-t$, thus $Y$ is a $(d, t)$-EKR set. Consider any generator $T$ with $\dim(U \cap T) < d - \frac{t}{2}$. Then $U$ has a subspace $R$ of dimension $d- \frac{t}{2}$ such that $\dim(R \cap U \cap T) < d-t$. The subspace $T' := \langle R, R^\bot \cap T\rangle$ is a generator on $R$ and in the quotient on $R$ one sees that there exists a generator $T''$ on $R$ with $T'' \cap T' = R$. Then
 \begin{align*}
  T'' \cap T &= T'' \cap T \cap R^\bot\\
  &= T'' \cap T \cap T' = T \cap R.
 \end{align*}
 Hence, $\dim(T \cap T'') < d-t$. As $T'' \in Y$, this shows that $Y \cup \{ T \}$ is not a $(d, t)$-EKR set and hence $Y$ is maximal.
\end{proof}

\begin{satz}\label{thm_max_t_even}
 Let $Y$ be a maximal $(d, t)$-EKR set with $d\ge 2t$ and $|Y| > b_1^\even + b_2^\even$. Then $Y$ is as in Example \ref{lem_ex_t_even}.
\end{satz}
\begin{proof}
 In the view of Lemma \ref{lem_upperbound_1_t_even} and Lemma \ref{lem_collected_properties_t_even} (a) there are distinct elements $a_1, a_2, a_3 \in Y$ such that $P := a_1 \cap a_2 \cap a_3$ has dimension $d-\frac{3}{2} t$. Put $U := \langle a_1 \cap a_2, a_1 \cap a_3, a_2 \cap a_3\rangle$. Lemma \ref{lem_collected_properties_t_even} gives $\dim(U) = d$ and shows that every $b \in Y$ satisfies $\dim(b \cap U) \geq d - \frac{t}{2}$ or $\dim(b \cap P) > d-2t$. If $\dim(b \cap U) \geq d- \frac{t}{2}$ for all $b \in Y$, then the maximality of $Y$ implies that $Y$ is Example \ref{lem_ex_t_even}. Otherwise, Lemma \ref{lem_t_even_P_large2} shows that $|Y| \leq b_1^\even + b_2^\even$.
\end{proof}

The following result was already shown by Brouwer and Hemmeter in \cite{MR1158800} for $\epsilon \neq \frac{1}{2}, \frac{3}{2}$.

\begin{kor}\label{kor_max_t_2}
  Let $Y$ be an $(d, 2)$-EKR set with $d\ge 4$ of maximum size (for fixed type $e$).
  Then either $Y$ is as in Example \ref{lem_ex_t_even} or all elements of $Y$ contain a fixed $(d-3)$-space.
\end{kor}
\begin{proof}
  In this case $b_2^\even = 0$, since
  \begin{align*}
    \psi^\even = 0.
  \end{align*}
  Hence, in the proof of Lemma \ref{lem_t_even_P_large2} $|Y_1| = 0$.
  Therefore either all elements of $Y$ contain the $(d-3)$-space $P$ or $Y$ is as in Example \ref{lem_ex_t_even} by the proof of Theorem \ref{thm_max_t_even}.
\end{proof}

\section{EKR Sets, $t$ odd}\label{sec_t_odd}

Throughout this section we assume that we work in a finite classical polar space of rank $d > 2$ and given type $\epsilon$. We assume throughout this section that $t\ge 3$ is an odd integer satisfying $d\ge 2t-1$. Recall that the case $t=1$ is covered by Lemma \ref{lem_d_t_eq_d_1}.

\begin{defi} Define the constants $b_1^\odd$, $b_2^\odd$, $b_3^\odd$ by
\begin{align*}
  b_1^\odd = &\gauss{d-\frac{3}{2}t+\frac{1}{2}}{(t-3)/2} c_{2t-2, t},\\
  b_2^\odd = & \omega(d, (t+1)/2) \psi^\odd,\\
  b_3^\odd = & q^{\frac{t-1}{2} \epsilon + \binom{(t-1)/2}{2}} \overline{\psi}^\odd.
\end{align*}
\end{defi}

\begin{lemma}\label{lem_bounds_t_odd}
Let $Y$ be a $(d, t)$-EKR set.
  \begin{enumerate}[(a)]
   \item Let $P$ be a totally isotropic subspace of dimension at least $d-\frac{3}{2}t+\frac{1}{2}$.
  If $\dim(c \cap P) \geq \dim(P) - \frac{t}{2} + \frac{3}{2}$ for all elements $c\in Y$, then $|Y|\le b_1^\odd$.
  \item Let $U$ be a totally isotropic subspace of dimension $d-1$, $P$ a subspace of $U$ of dimension $d - \frac{3}{2} t + \frac{1}{2}$, and $A$ a subspace of $P$ with $\dim(A) \geq d - 2t + 2$. If all elements $c\in Y$ satisfy $\dim(U \cap c) = d - \frac{t}{2} - \frac{1}{2}$, $\dim(c \cap P) = d - 2t + 1$, and $\dim(c \cap A) \geq \dim(A) - \frac{t}{2} + \frac{3}{2}$, then $|Y|\le b_2^\odd$.
  \item Let $G$ be a generator, and $P$ a subspace of $G$ of dimension $d - \frac{3}{2} t + \frac{1}{2}$.
  If all $c\in Y$ satisfy $\dim(G \cap c) = d - \frac{t}{2} + \frac{1}{2}$ and $\dim(c \cap P) = d - 2t + 1$, then $|Y|\le b_3^\odd$.\label{lem_bounds_t_odd_c}
  \end{enumerate}
\end{lemma}
\begin{proof}
 \begin{enumerate}[(a)]
  \item By replacing $P$ if necessary by a subspace of dimension $d - \frac{3}{2} t + \frac{1}{2}$, we may assume that $\dim(P) = d - \frac{3}{2} t + \frac{1}{2}$. Then the Gaussian coefficient in the definition of $b_1^\odd$ is the number of subspaces $U$ of $P$ of codimension $\frac{t}{2} - \frac{3}{2}$. By hypothesis, every element of $Y$ contains one such subspace $U$. The element of $Y$ on such a fixed subspace $U$ form a $(2t-2, t)$-EKR set in the quotient geometry on $U$.
  \item By replacing $A$ if necessary by a subspace of $A$ of dimension $d-2t+2$, we may assume that $\dim(A) = d -2t+2$. There are $\psi^\odd$ subspaces $T$ of $U$ with $\dim(T) = d - \frac{t}{2} - \frac{1}{2}$, $\dim(T \cap P) = d - 2 t + 1$, and $\dim(T \cap A) = \dim(A) - \frac{t}{2} + \frac{1}{2} + i$ with $i \in \{ 1, \ldots, t/2-3/2\}$.
  For each such $T$, there are exactly $\omega(d, (t+1)/2)$ generators $a$ with $a \cap U \supseteq T$.
  \item There are exactly $\overline{\psi}^\odd$ subspaces $T$ of $G$ with $\dim(T) = d - \frac{t}{2} + \frac{1}{2}$ and $\dim(T \cap P) = d - 2t + 1$. For each such $T$, there are exactly $q^{\frac{t-1}{2} \epsilon + \binom{(t-1)/2}{2}}$ generators $a$ with $T = a \cap G$, since it is well-known that $q^{\frac{t-1}{2} \epsilon + \binom{(t-1)/2}{2}}$ generators are disjoint to $G$ in the quotient geometry of $T$ (see for example Corollary \ref{kor_dis_gens}).
 \end{enumerate}
\end{proof}

\begin{lemma}\label{lem_collected_properties_t_odd}
 Let $Y$ be a $(d, t)$-EKR set, and consider $a_1, a_2, a_3 \in Y$. Then the following holds true.
 \begin{enumerate}[(a)]
  \item The dimension of $a_1 \cap a_2 \cap a_3$ is at least $d - \frac{3}{2} t + \frac{1}{2}$. \label{lem_intersection_odd_a}
  \item Suppose that equality holds in \eqref{lem_intersection_odd_a} and put $\ell_{ij} = a_i \cap a_j$ for different $i, j \in \{ 1, 2, 3\}$, $U := \langle \ell_{12}, \ell_{13}, \ell_{23} \rangle$, and $P = a_1 \cap a_2 \cap a_3$.
  Then one of the following cases occurs:
    \begin{enumerate}[1.]
    \item \begin{enumerate}[(i)]
      \item $\dim(U) = d-1$.
      \item $\dim(\ell_{ij}) = d-t$ for $1 \leq i < j \leq 3$.
      \item $a_i \cap U = \langle \ell_{ij}, \ell_{ik} \rangle$ for $\{ i,j,k\} = \{1, 2, 3\}$.
      \item $\dim(a_i \cap U) = d - \frac{t}{2} - \frac{1}{2}$ for $i\in \{1, 2, 3\}$.
      \item Every $b \in Y$ satisfies $\dim(b \cap P) \ge d - 2t + 1$ and equality implies that  $\dim(b \cap U) \ge d - \frac{t}{2} - \frac{1}{2}$.
    \end{enumerate}
    \item \begin{enumerate}[(i)]
      \item $\dim(U) = d$.
      \item $\dim(\ell_{ij}) = \dim(\ell_{ik}) = d-t$ and $\dim(\ell_{jk}) = d-t+1$ for some $\{ i,j,k\} = \{1, 2, 3\}$.
      \item $a_i \cap U = \langle \ell_{ij}, \ell_{ik} \rangle$ for $\{ i,j,k\} = \{1, 2, 3\}$.
      \item $\dim(a_j \cap U) = \dim(a_k \cap U) = d - \frac{t}{2} + \frac{1}{2}$, and $\dim(a_i \cap U) = d - \frac{t}{2} - \frac{1}{2}$ for some $\{ i,j,k\} = \{1, 2, 3\}$ (with the same order as in (ii)).
      \item Every $b \in Y$ satisfies $\dim(b \cap P) \ge d - 2t + 1$ and equality implies that  $\dim(b \cap U) \ge d - \frac{t}{2} - \frac{1}{2}$.
      Also, if $\dim(b \cap P) = d - 2t +1$ and $\dim(a_i\cap a_j)=d-t$, then $\dim(a_i \cap a_j \cap b) = d - \frac{3}{2}t + \frac{1}{2}$.
    \end{enumerate}
    \end{enumerate}
 \end{enumerate}
\end{lemma}
\begin{proof}
  \begin{enumerate}[(a)]
   \item \label{lemma_a123_a_prf}As $Y$ is a $(d, t)$-EKR set, then $\dim(\ell_{ij}) \geq d-t$. As $U$ is totally isotropic, then $\dim(U) \leq d$. Since $\ell_{12} \cap \ell_{13} = P$ as well as $\langle \ell_{12}, \ell_{13} \rangle \cap \ell_{23} = P$ (because $\langle \ell_{12}, \ell_{13} \rangle \cap \ell_{23} \subseteq a_1 \cap \ell_{23} = P$), then
   \begin{align*}
      d &\geq \dim(U) = \dim(\langle \ell_{12}, \ell_{13}, \ell_{23} \rangle)\\
      &= \dim( \langle \ell_{12}, \ell_{23} \rangle) + \dim(\ell_{23}) - \dim(P)\\
      &= \dim(\ell_{12}) + \dim(\ell_{13}) - \dim(P) + \dim(\ell_{23}) - \dim(P)\\
      &\geq 3(d-t) - 2\dim(P).
   \end{align*}
   Hence, $\dim(P) \geq d - \frac{3}{2}t + \frac{1}{2}$.
   \item As $\dim(P) = d - \frac{3}{2} t + \frac{1}{2}$, the argument in \eqref{lemma_a123_a_prf} shows that $\dim(U) \in \{ d-1, d\}$.

   Consider first the case that $\dim(U) = d-1$. Then the argument to prove the first part of the lemma yields $\dim(\ell_{ij}) = d-t$ for all $i\not=j$.
   For $\{ i, j, k\} = \{ 1, 2, 3\}$, we have $\ell_{ij}, \ell_{ik} \subseteq a_i \cap U$ and hence $U = \langle a_i \cap U, \ell_{jk} \rangle$.
   Also $a_i \cap U \cap \ell_{jk} = a_i \cap \ell_{jk} = P$. Hence,
   \begin{align*}
    d-1 &= \dim(U) = \dim(\langle a_i \cap U, \ell_{jk} \rangle)\\
    & = \dim(a_i \cap U) + \dim(\ell_{jk}) - \dim(P)\\
    &= \dim(a_i \cap U) + (d-t) - (d- \frac{3}{2}t + \frac{1}{2}).
   \end{align*}
   Therefore, $\dim(a_i \cap U) = d - \frac{t}{2} - \frac{1}{2}$, which implies that $a_i \cap U = \langle \ell_{ij}, \ell_{ik} \rangle$.
   We have proved the first four statements for the case $\dim(U)=d-1$. The arguments for these for the corresponding statements in the case $\dim(U)=d$ are similar and omitted. The final part is proved for both cases together.

   Consider $b \in Y$. We may assume that $\dim(a_1\cap a_2)=d-t$. It follows from the first statement of the lemma that $\ell_{ij}' := b \cap \ell_{ij}$, $1\le i<j\le 3$ has dimension at least $d-\frac{3}{2} t + \frac{1}{2}$. As $P$ and $\ell'_{12}=b \cap a_1 \cap a_2$ lie in $a_1 \cap a_2$, the dimension formula shows that
 \begin{align*}
  \dim(b \cap P) &= \dim(\ell'_{12} \cap P)\\
  &\geq \dim(\ell'_{12}) + \dim(P) - \dim(a_1 \cap a_2)\\
  &\geq 2(d-\frac{3}{2} t + \frac{1}{2})-(d-t)=d-2t+1,
 \end{align*}
 and equality implies that $\ell'_{12} = a_1 \cap a_2 \cap b$ has dimension $d-\frac{3}{2}t+\frac{1}{2}$.
 Suppose finally that  $\dim(b \cap P)=d-2t+1$. We have $\langle \ell_{12}', \ell_{13}'\rangle \cap \ell_{23}' \subseteq (a_1 \cap b) \cap \ell_{23} = b \cap P$ and $\ell_{12}' \cap \ell_{13}' = b \cap P$. This implies that
   \begin{align*}
      \dim(b \cap U) &\geq \dim( \langle \ell_{12}', \ell_{13}', \ell_{23}' \rangle)\\
      &\geq  \dim(\ell_{12}') + \dim(\ell_{13}') + \dim(\ell_{23}') - 2\dim(b \cap P)\\
      &\geq 3(d - \frac{3}{2}t + \frac{1}{2}) - 2\dim(b \cap P) =d - \frac{t}{2} - \frac{1}{2}.
   \end{align*}
  \end{enumerate}
\end{proof}

\begin{lemma}\label{lem_t_odd_P_large}
 If $\dim(a_1 \cap a_2 \cap a_3) > d - \frac{3}{2} t + \frac{1}{2}$ for all $a_1, a_2, a_3$ of a maximal $(d, t)$-EKR set $Y$, then $|Y| \leq b_1^\odd$.
\end{lemma}
\begin{proof}
 Lemma \ref{lem_max_dmtp1_is_not_max_dmt_ekr_set} gives $a_1, a_2 \in Y$ with $\dim(a_1 \cap a_2) = d-t$.
 Consider a third element $a_3 \in Y$ and put $P := a_1 \cap a_2 \cap a_3$. Consider any element $b \in Y$.
 By hypothesis, $\dim(b \cap a_1 \cap a_2) \geq d - \frac{3}{2}t + \frac{3}{2}$.
 As $P$ and $b \cap a_1 \cap a_2$ lie in $a_1 \cap a_2$, the dimension formula shows that
 \begin{align*}
  \dim(b \cap P) &= \dim((b \cap a_1 \cap a_2) \cap P)\\
  &\geq  \dim(b \cap a_1 \cap a_2) + \dim(P) - \dim(a_1 \cap a_2)\\
  &\geq \dim(P) - \frac{t}{2} + \frac{3}{2}.
 \end{align*}
  As $\dim(P) > d - \frac{3}{2} t - \frac{1}{2}$ (by hypothesis), Lemma \ref{lem_bounds_t_odd} proves the assertion.
\end{proof}

\begin{lemma}\label{lem_t_odd_P_large2}
 Let $Y$ be a $(d, t)$-EKR set such that there exist a generator $G_0$, a $(d-1)$-space $U \subseteq G_0$, and $(d - \frac{3}{2}t + \frac{1}{2})$-spaces $P, Q \subseteq U$ such that $a \in Y$ implies the following:
 \begin{align*}
  &\dim(a \cap U) \geq d - \frac{t}{2} - \frac{1}{2} \text{ or }\\
  &\dim(a \cap P) > d - 2t + 1 \text{ or } \dim(a \cap Q) > d - 2t + 1.
 \end{align*}
 Suppose also that $\dim( a \cap U) < d -\frac{t}{2} - \frac{1}{2}$ for at least one element $a$ of $Y$. Then $|Y| \leq 2b_1^\odd + b_2^\odd + b_3^\odd$.
\end{lemma}
\begin{proof}
  Define
  \begin{align*}
    &Y_1 := \{ a \in Y: \dim(a \cap U) = d-\frac{t}{2}-\frac{1}{2}, ~\dim(a \cap P) = \dim(a \cap Q) = d-2t+1 \},\\
    &Y_2 := \{ a \in Y: \dim(a \cap P) > d-2t+1 \},\\
    &Y_3 := \{ a \in Y: \dim(a \cap Q) > d-2t+1 \}.
  \end{align*}
  We claim first that $Y = Y_1 \cup Y_2 \cup Y_3$. To see this, let $a\in Y$. Then
 \begin{align*}
 \dim(a\cap U)&=\dim(\langle a \cap U, P \rangle)+\dim(a\cap P)-\dim(P)
 \\
   &\leq \dim(U)+\dim(a\cap P)-\dim(P)
 \\
   &\leq \dim(a\cap P)+\frac32t-\frac32
 \end{align*}
 and similarly $ \dim(a\cap U)\le \dim(a\cap Q)+\frac32t-\frac32$. Hence, if $a\notin Y_2\cup Y_3$, then the hypothesis of the lemma implies that $a\in Y_1$. Thus we have proved that $Y = Y_1 \cup Y_2 \cup Y_3$.

 Lemma \ref{lem_bounds_t_odd} gives $|Y_2|+|Y_3| \leq 2b_1^\odd$.

  By hypothesis, there exists an element $a_2 \in Y_2\cup Y_3$ with $\dim(a_2 \cap U) \leq d- \frac{t}{2} - \frac{3}{2}$.
  By symmetry, we may assume that $a\in Y_2$. Define the following two subsets of $Y_1$:
  \begin{align*}
    &S := \{ a_1 \in Y_1: \langle a_1 \cap P, a_2 \cap P \rangle = P \}\\
    &T := \{ a_1 \in Y_1: \langle a_1 \cap P, a_2 \cap P \rangle \neq P \}.
  \end{align*}
  In the following, we use Lemma \ref{lem_bounds_t_odd} to show that $|S| \leq b_3^\odd$ and $|T| \leq b_2^\odd$.

  Let $a_1 \in S$.
  As $\dim(U) - \dim(a_1 \cap U) = \frac{t}{2} - \frac{1}{2} = \dim(P) - \dim(a_1 \cap P)$
  we have $U = \langle a_1 \cap U, P \rangle$. As $P = \langle a_1 \cap P, a_2 \cap P \rangle$
  this implies that $U = \langle a_1 \cap U, a_2 \cap P\rangle$ and hence
  $U = \langle a_1 \cap U, a_2 \cap U\rangle$. Therefore every point of $a_1\cap a_2$ lies in $U^\bot$ and
  \begin{align*}
   \dim(a_1\cap a_2\cap U)&=\dim(a_1\cap U)+\dim(a_2\cap U)-\dim(U)\\
   &\le (d- \frac{t}{2} - \frac{1}{2})+(d- \frac{t}{2} - \frac{3}{2})-(d-1) \\
   &=d-t-1.
  \end{align*}
   By $a_1, a_2 \in Y$, we have $\dim(a_1\cap a_2)\ge d-t$.
   As $U$ has dimension $d-1$, it follows that $U$ and $a_1\cap a_2$ span a generator,
   which implies that $\dim(a_1\cap a_2)= d-t$ and $\dim(a_1\cap a_2\cap U)=d-t-1$,
   which in turn shows that $\dim(a_2 \cap U) = d-\frac{t}{2}-\frac{3}{2}$.
   Then $\dim(a_1\cap G)=\dim(a_1\cap U)+1=d-\frac{t}{2}+\frac{1}{2}$.
   Clearly, $G=\langle U,U^\perp\cap a_2\rangle$ and thus $G$ is independent of the
   choice of $a_1\in S$. Hence every element of $S$ meets $G$ in a subspace of dimension $d-\frac{t}{2}+\frac{1}{2}$.
   Recall $S \subseteq Y_1$, so $\dim(a_1 \cap P) = d-2t+1$.
  Applying the third part of Lemma \ref{lem_bounds_t_odd} now gives
  \begin{align*}
    |S| \leq b_3^\odd.
  \end{align*}

  For $a_1 \in T$ we have that $\langle a_1 \cap P, a_2 \cap P \rangle$ is a proper
  subspace of $P$ and thus we can improve the previous estimate to
  $\dim(a_1 \cap a_2 \cap P) \geq \dim(a_2 \cap P) - \frac{t}{2} + \frac{3}{2}$.
  Then the second part of Lemma  \ref{lem_bounds_t_odd} again applied with $A = a_2 \cap P$
  gives $|T|\le b_2^\odd$. Hence $|Y| \le |Y_2|+|Y_3|+|S|+|T| \le  2b_1^\odd + b_2^\odd + b_3^\odd$.
\end{proof}

\begin{beispiel}\label{lem_ex_t_odd}
 The set consisting of all generators which meet a given $(d-1)$-space $U$ in a subspace of dimension at least $d - \frac{t}{2} - \frac{1}{2}$ is a maximal $(d, t)$-EKR set.
\end{beispiel}
\begin{proof}
 As the given subspace $U$ has dimension $d-1$, the dimension formula shows that
 the meet any of two elements of $Y$ has dimension at least $d-t$, thus $Y$ is a $(d, t)$-EKR set.
 Consider any generator $T$ with $\dim(U \cap T) < d- \frac{t}{2} - \frac{1}{2}$.
 Then $G$ has a subspace $R$ of dimension $d - \frac{t}{2}-\frac{1}{2}$ such that $\dim(R \cap G \cap T) < d-t$.
 The subspace $T' := \langle R, R^\bot \cap T \rangle$ is a generator on $R$ and in the quotient on $R$ one sees that there exists a generator $T''$ on $R$ with $T'' \cap T' = R$.
 Then $T'' \cap T = R \cap G \cap T$ and hence $\dim(T \cap T'') < d-t$. As $T'' \in Y$, this shows that $Y \cup \{ T \}$ is not a $(d, t)$-EKR set.
\end{proof}

We write $P_{ijk}$ for $a_i \cap a_j \cap a_k$ and $U_{ijk}$ for $\langle a_i \cap a_j, a_i \cap a_k, a_j \cap a_k \rangle$
in the remaining parts of this section.
Hereby we are allowed to substitute $i$, $j$, or $k$ with other symbols.
This is a purely formal convention. Each string $P_{ijk}$ or $U_{ijk}$ is only
an expression if $a_i$, $a_j$, and $a_k$ are appropriately defined.

\begin{lemma}\label{lem_t_odd_intersection_123_124}
  Let $Y$ be a $(d, t)$-EKR set.
  Let $a_1, a_2, a_3, a_4 \in Y$.
  Suppose that we have
  \begin{align*}
    \dim(P_{123}) = \dim(P_{124}) = d - \frac{3}{2} t + \frac{1}{2},
  \end{align*}
  $\dim(a_4 \cap P_{123}) = d-2t+1$, and $\dim(a_1 \cap a_2) = d-t$.
  Let $U = U_{123} \cap U_{124}$. Then
  \begin{align*}
    \dim(U) \geq d-1 \text{ and } \dim(a_4 \cap U) \geq \dim(U) - \frac{t}{2} + \frac{1}{2}.
  \end{align*}
\end{lemma}
\begin{proof}
  By Lemma \ref{lem_collected_properties_t_odd} (a), $\dim(P_{ijk}) \geq d - \frac{3}{2} t + \frac{1}{2}$.
  Hence,
  \begin{align*}
   \dim(U) &\geq \dim(\langle a_1\cap a_2,a_1\cap a_3\cap a_4,a_2\cap a_3\cap a_4 \rangle) \\
   &\geq \dim(a_1\cap a_2)+\dim(\langle a_1\cap a_3\cap a_4,a_2\cap a_3\cap a_4\rangle)  -\dim(a_4\cap P_{123})\\
   &= (d-t) + \dim(a_1\cap a_3\cap a_4) + \dim(a_2\cap a_3\cap a_4) - 2\dim(a_4\cap P_{123})\\
   &\geq (d-t) + 2(d - \frac{3}{2} t + \frac{1}{2}) -2\dim(a_4\cap P_{123}) = d-1.
  \end{align*}
  This shows the first part of the assertion.

   If $\dim(U_{124}) = d-1$, then the claim follows by Lemma \ref{lem_collected_properties_t_odd} (b) 1.(iv).
   Hence suppose $\dim(U_{124}) = d$.
   Then Lemma \ref{lem_collected_properties_t_odd} (b) 2.(ii) shows that
    $\dim(a_1 \cap a_4) + \dim(a_2 \cap a_4) = 2d-2t+1$.
    Then, by Lemma \ref{lem_collected_properties_t_odd} (b) 2.(iv), $\dim(a_4 \cap U_{124}) = d - \frac{t}{2} + \frac{1}{2}$.
    Hence, $\dim(a_4 \cap U) \geq d - \frac{t}{2} - \frac{1}{2}$.
\end{proof}

\begin{lemma}\label{lem_t_odd_funny_case}
 If for all $a_1, a_2, a_3$ of a maximal $(d, t)$-EKR set $Y$, which is not as in Example \ref{lem_ex_t_odd},
 \begin{align*}
  \dim(a_1 \cap a_2 \cap a_3) = d - \frac{3}{2} t + \frac{1}{2}
 \end{align*}
 implies $\dim(U_{123}) = d$, then $|Y| \leq 2 b_1^\odd + b_2^\odd + b_3^\odd$.
\end{lemma}
\begin{proof}
  If $\dim(a_1 \cap a_2 \cap a_3) > d- \frac{3}{2}t + \frac{1}{2}$ for all $a_1, a_2, a_3 \in Y$,
  then Lemma \ref{lem_t_odd_P_large} shows $|Y| \leq b_1^\odd$.
  Hence suppose that there are $a_1, a_2, a_3 \in Y$ with $\dim(a_1 \cap a_2 \cap a_3) = d- \frac{3}{2}t + \frac{1}{2}$,
  $\dim(a_1 \cap a_2) = d-t$, and $\dim(a_1 \cap U_{123}) = d-\frac{t}{2}-\frac{1}{2}$
  (see Lemma \ref{lem_collected_properties_t_odd} (b) 2.).
  Set
  \begin{align*}
    &Y_1 := \{ a \in Y: \dim(a \cap P_{123}) > d-2t+1\}, \\
    &Y_2 := \{ a \in Y: \dim(a \cap P_{123}) = d-2t+1\}.
  \end{align*}
  By Lemma \ref{lem_collected_properties_t_odd} (b) 2.(v), $Y = Y_1 \cup Y_2$ is a partition of $Y$.
  We have $|Y_1| \leq b_1^\odd$ by Lemma \ref{lem_bounds_t_odd}.
  We may thus assume that $Y_2 \neq \emptyset$.

  \textbf{Case 1. All $a_4 \in Y_2$ satisfy $\dim(U_{123} \cap U_{124}) = d$.}
  Let $U \subseteq U_{123}$ be a $(d-1)$-dimensional subspace of $U_{123}$ with
  $a_1 \cap U_{123} \subseteq U$.
  By Lemma \ref{lem_collected_properties_t_odd} (b) 2.(v) and $\dim(a_4 \cap P_{123}) = d-2t+1$ show
  $\dim(P_{124}) = d - \frac{3}{2} t + \frac{1}{2}$ for all $a_4\in Y_2$.
  By Lemma \ref{lem_collected_properties_t_odd} (b) 2. (iv) and $\dim(a_1 \cap a_2) = d-t$, all $a_4 \in Y_2$ satisfy
  \begin{align*}
    \dim(a_4 \cap U) &\geq \dim(a_4 \cap U_{124})-1 \\
    &= d - \frac{t}{2} - \frac{1}{2}.
  \end{align*}
  As $Y$ is not as in Example \ref{lem_ex_t_odd}, there exists an $a_5 \in Y$ with $\dim(a_5 \cap U) < d - \frac{t}{2} - \frac{1}{2}$.
  Hence, we can apply Lemma \ref{lem_t_odd_P_large2} with $G_0 = U_{123}$ and $P=Q=P_{123}$. This shows $|Y| \leq 2b_1^\odd + b_2^\odd + b_3^\odd$.

  \textbf{Case 2. There exists a generator $a_4 \in Y_2$ with $\dim(U_{123} \cap U_{124}) \leq d-1$.}
  Put $U := U_{123} \cap U_{124}$.
  By Lemma \ref{lem_collected_properties_t_odd} (b) 2.(v) and $\dim(a_4 \cap P_{123}) = d-2t+1$,
  $\dim(P_{124}) = d - \frac{3}{2} t + \frac{1}{2}$.
  Hence, by Lemma \ref{lem_t_odd_intersection_123_124}, $\dim(U) = d-1$.
  Define the following subsets of $Y_2$:
  \begin{align*}
    S := \{ a_i \in Y_2: \dim(a_i \cap P_{124}) = d-2t+1\},\\
    T := \{ a_i \in Y_2: \dim(a_i \cap P_{124}) > d-2t+1 \}.
  \end{align*}
  By Lemma \ref{lem_collected_properties_t_odd} (b) 2.(v), this is a partition $Y_2 = S \cup T$ of $Y_2$.
  Let $a_i \in S$. By Lemma \ref{lem_collected_properties_t_odd} (b) 2.(v) and $\dim(a_i \cap P_{123}) = d-2t+1$,
  $\dim(P_{12i}) = d - \frac{3}{2} t + \frac{1}{2}$.
  Hence, by Lemma \ref{lem_t_odd_intersection_123_124}, $\dim(U_{12i} \cap U_{123}) \geq d-1$
  and $\dim(U_{12i} \cap U_{124}) \geq d-1$.
  Suppose for a contradiction $\dim(U_{12i} \cap U_{123} \cap U_{124}) = d-2$.
  Then $\dim(U_{12i} \cap U_{123}) = d-1$ and $\dim(U_{12i} \cap U_{124}) = d-1$.
  Hence,
  \begin{align*}
    d &\geq \dim(\langle U_{12i} \cap U_{123}, U_{12i} \cap U_{124}, U_{123} \cap U_{124} \rangle) \\
    &= \dim(U_{12i} \cap U_{123}) + \dim(U_{12i} \cap U_{124}) + \dim(U_{123} \cap U_{124}) - 2\dim(U_{12i} \cap U_{123} \cap U_{124})\\
    &= 3(d-1) - 2(d-2) = d+1.
  \end{align*}
  This is a contradiction. Hence, we have $U \subseteq U_{123} \cap U_{12i}$.
  Hence, by Lemma \ref{lem_t_odd_intersection_123_124},
  all $a_i \in S$ satisfy
  \begin{align*}
    \dim(a_i \cap U) \geq d - \frac{t}{2} - \frac{1}{2}.
  \end{align*}
  As $Y$ is not as in Example \ref{lem_ex_t_odd}, there exists an $a_5 \in Y$ with $\dim(a_5 \cap U) < d - \frac{t}{2} - \frac{1}{2}$.
  Thus we can apply Lemma \ref{lem_t_odd_P_large2} with $G_0 = U$, $P=P_{123}$, and $Q=P_{124}$.
  This shows $|Y| \leq 2b_1^\odd + b_2^\odd + b_3^\odd$.
\end{proof}

\begin{satz}\label{thm_max_t_odd}
 Let $Y$ be a maximal $(d, t)$-EKR set where $|Y| > 2b_1^\odd + b_2^\odd + b_3^\odd$. Then $Y$ is as in Example \ref{lem_ex_t_odd}.
\end{satz}
\begin{proof}
  In the view of Lemma \ref{lem_t_odd_P_large} we may assume that $Y$ has distinct
  elements $a_1$, $a_2$, $a_3$ such that $P := a_1 \cap a_2 \cap a_3$ has dimension
  $d - \frac{3}{2} t + \frac{1}{2}$. Put $U = \langle a_1 \cap a_2, a_1 \cap a_3, a_2 \cap a_3 \rangle$.
  We may suppose $\dim(U) = d-1$ by Lemma \ref{lem_collected_properties_t_odd} and
  Lemma \ref{lem_t_odd_funny_case}. Lemma \ref{lem_collected_properties_t_odd} shows
  that every $b \in Y$ satisfies $\dim(b \cap U) \geq d- \frac{t}{2} - \frac{1}{2}$
  or $\dim(b \cap P) > d-2t+1$. If $\dim(b \cap U) \geq d - \frac{t}{2} - \frac{1}{2}$
  for all $b \in Y$, then the maximality of $Y$ implies that $Y$ is as in Example \ref{lem_ex_t_odd}.
  Otherwise Lemma \ref{lem_t_odd_P_large2} shows that $|Y| \leq 2b_1^\odd + b_2^\odd + b_3^\odd$.
\end{proof}

\section{Some Inequalities}\label{sec_ests}

We will need some upper and lower estimates for the number of generators in a polar space and for the Gaussian coefficients.

\begin{lemma}[{\cite{MR2349596}}]\label{lem_log_bound}
  Let $x \geq 0$. Then we have
  \begin{align*}
    \frac{2x}{2+x} \leq \log(1+x) \leq \frac{x}{2} \cdot \frac{2+x}{1+x}.
  \end{align*}
\end{lemma}

\begin{lemma}\label{lem_log_quotient}
  Let $q \geq 2$ and let $f: [0, \infty) \rightarrow \bbR$ be the function
  \begin{align*}
    f(x) = \frac{\log(1+q^{-x})}{\log(1+q^{-x-1})}.
  \end{align*}
  Then the first derivative $f'$ of $f$ is bounded by
  \begin{align*}
    f'(x) \geq \frac{(q^x (2q-2) - 1)\log(q)}{2q^{x+1}(2q^x+1)(q^x+1) (q^{x+1}+1) \log(1 + q^{-x-1})^2 }
  \end{align*}
  In particular, $f$ is monotonically increasing in $x$, i.e. $f'(x) > 0$ for all $x \in (0, \infty)$.
\end{lemma}
\begin{proof}
  We have 
  \begin{align*}
   f'(x) =\frac{\left( (q^x+1) \log(1+q^{-x}) - (q^{x+1} + 1) \log( 1 + q^{-x-1}) \right)\log(q)}{(q^x+1) (q^{x+1}+1)\log(1 + q^{-x-1})^2}\\
  \end{align*}
  and Lemma \ref{lem_log_bound} implies that 
 \begin{align*}
 & (q^x+1) \log(1+q^{-x}) - (q^{x+1} + 1) \log( 1 + q^{-x-1})
 \\ &\ge (q^x+1)\frac{2q^{-x}}{2+q^{-x}}-(q^{x+1}+1)\cdot\frac{q^{-x-1}}{2}\cdot\frac{2+q^{-x-1}}{1+q^{-x-1}}
 \\&=\frac{2q^{x+1}-2q^x-1}{2q^{x+1}(2q^x+1)}.
 \end{align*}
 
\end{proof}

\begin{kor}\label{lem_log_quotient2}
Define the functions $g,\alpha: [0, \infty) \times [2, \infty) \rightarrow \bbR$ by
 \begin{align*}
    \alpha(x,q) & :=\frac{\log(1+q^{-x})}{\log(1+q^{-x-1})}
    \\
        g(x, q) &:= (1+q^{-x})^{\frac{\alpha(x,q)}{\alpha(x,q)-1}}
 \end{align*}
  For fixed $q$, the function $g$ is monotonically decreasing in $x$. Also the function $g(0, q)$ is monotonically decreasing in $q$.
\end{kor}
\begin{proof}
  For fixed $q\ge 2$, Lemma \ref{lem_log_quotient} shows that
  \begin{align*}
  \frac{\alpha(x,q)}{\alpha(x,q)-1} = 1 +\frac{1}{\alpha(x,q)-1}
  \end{align*}
  is monotonically decreasing in $x$. For fixed $q \geq 2$ also the function $x\mapsto 1+q^{-x}$ for $x\ge 0$ is monotonically decreasing. Hence for fixed $q$, the function $g$ is monotonically decreasing in $x$.
  The derivative of $\frac{\alpha(0,q)}{\alpha(0,q)-1}$ with respect to $q$ is
  \begin{align*}
    - \frac{\alpha(0,q)}{(\alpha(0,q)-1)^2 q (q+1) \log(1 + q^{-1})} < 0.
  \end{align*}
  As $1+q^{-x} = 2 > 1$ for $x=0$, this shows that $g(0, q)$ is monotonically decreasing in $q$.
\end{proof}

\begin{lemma}\label{lem_est_gens}
  Let $\scrP$ be a polar space of rank $d$ and type $e$.
  \begin{enumerate}[(a)]
  \item The polar space $\scrP$ contains at least
  \begin{align*}
     q^{de + \binom{d}{2}}
  \end{align*}
    generators.
  \item Let $\alpha > 1$ be a real number with
  \begin{align*}
    \alpha \log(1+q^{-e-1}) \leq \log(1+q^{-e}).
  \end{align*}
  Let $x$ be the number of generators of $\scrP$.
  Then
  \begin{align*}
    x \cdot q^{-de - \binom{d}{2}} = \prod_{i=0}^{d-1} (1+ q^{-e-i}) \leq \left( 1 + q^{-e} \right)^{\frac{\alpha}{\alpha-1}}.
  \end{align*}
  Furthermore, the second inequality holds for all $e \in \bbR$.
  \end{enumerate}
\end{lemma}
\begin{proof}
  The first claim is a trivial consequence of \eqref{number_gens}.
  We shall prove the second claim in the following.

  By Lemma \ref{lem_log_quotient}, the hypothesis on $\alpha$ implies
  \begin{align*}
   \alpha \log(1+q^{-e-1-i}) \leq \log(1+q^{-e-i})
  \end{align*}
  for all $i \geq 0$. Then
  \begin{align*}
    \sum_{i=0}^\infty \log(1+q^{-e-i}) &\leq \sum_{i=0}^\infty \alpha^{-i} \log(1+q^{-e})\\
    &= \log(1+q^{-e}) \sum_{i=0}^\infty \alpha^{-i} = \log(1+q^{-e}) \frac{\alpha}{\alpha - 1}.
  \end{align*}
  Hence,
  \begin{align*}
    \prod_{i=0}^{d-1} (1+ q^{-e-i}) \leq \left( 1 + q^{-e} \right)^{\frac{\alpha}{\alpha-1}}.
  \end{align*}
\end{proof}

Particularly, the upper bound in the previous result is noteworthy as it is much tighter
for many choices of $q$ and $e$ than the standard upper bound
  \begin{align*}
    \prod_{i=0}^{d-1} (1+ q^{-e-i}) < 2 + \frac{1}{q^e},
  \end{align*}
which holds for $e\ge \frac12$ and $q^e\ge 2$. See \cite[Lemma 11]{MR2755082} for a proof of this standard bound.

\begin{lemma}\label{lem_upper_bnd_gauss}
  Let $n \geq k \geq 0$.
  \begin{enumerate}[(a)]
   \item Let $q \geq 3$. Then
   \begin{align*}
      \gauss{n}{k} \leq 2 q^{k(n-k)}.
   \end{align*}
   \item Let $q \geq 4$. Then
  \begin{align*}
    \gauss{n}{k} \leq (1 + 2q^{-1}) q^{k(n-k)}.
  \end{align*}
   \item Let $q \geq 2$. Let $n \geq 1$. Then
  \begin{align*}
    \gauss{n}{1} \leq \frac{q}{q-1} q^{n-1}.
  \end{align*}
  \end{enumerate}
\end{lemma}
\begin{proof}
  Part (c) follows from the definition of the Gaussian coefficient. We have
  \begin{align}
    \gauss{n}{k} &= \prod_{i=1}^k \frac{q^{n-k+i}-1}{q^i-1} \leq \prod_{i=1}^k \frac{q^{n-k+i}}{q^i-1} = q^{k(n-k)} \prod_{i=1}^k \frac{q^i}{q^i-1}.\label{lem_upper_gauss_ineq1}
  \end{align}
  For (a) and (b) we therefore have to show that
  \begin{align*}
   \prod_{i=1}^k \frac{q^i}{q^i-1} \leq 1 + \alpha q^{-1},
  \end{align*}
  with $\alpha=3$ for $q=3$ and $\alpha=2$ for $q\geq 4$. This can easily be checked by hand for $k\le 2$. For $k\ge 3$ we use induction on $k$ to prove the stronger statement
  \begin{align}
    \prod_{i=1}^k \frac{q^i}{q^i-1} \leq 1 + \alpha \frac{q^{k-1}-2}{q^{k}-2}.  \label{lem_upper_gauss_ind}
  \end{align}
  For $k=3$, this is easily verified. The induction step follows from
  \begin{align*}
  \left( 1 + \alpha \frac{q^{k-1}-2}{q^{k}-2} \right) \frac{q^{k+1}}{q^{k+1}-1}
    = & 1 + \alpha \frac{q^{k}-2}{q^{k+1}-2} - \\
    &\begin{cases}
	    	     \frac{2(3^{2k+2} + 10 \cdot 3^k - 8)}{(3^k-2)(3^{k+1} - 2)(3^{k+1}-1)} & \text{ if } q = 3 \text{ and } \alpha = 3,\\
             \frac{(q^{k+1}-2q^k+2)(4q^{k+1} -q^k -6)}{(q^k-2)(q^{k+1} - 2)(q^{k+1}-1)} & \text{ if } q\geq 4 \text{ and } \alpha = 2.
    \end{cases}
  \end{align*}
  for $k\ge 3$.
\end{proof}

\begin{lemma}\label{lem_lowest_gauss}
  For integers $n>k>0$ we have
  \begin{align*}
    (1+q^{-1}) q^{k(n-k)} \leq \gauss{n}{k}.
  \end{align*}
\end{lemma}
\begin{proof}
  We have
  \begin{align*}
   \gauss{n}{k} &= \prod_{i=1}^k \frac{q^{n-k+i}-1}{q^i-1}
   \geq \frac{q^{n+1-k}-1}{q-1} \cdot q^{(k-1)(n-k)}\\
   &\geq (q^{n-k}+q^{n-k-1}) \cdot q^{(k-1)(n-k)}.
  \end{align*}
  This proves the statement.
\end{proof}

\section{The Association Scheme of a Dual Polar Graph}\label{sec_assoc}

We need some basic properties of association schemes of the dual polar graphs of rank $d$.
A complete introduction to association schemes can be found in \cite[Ch. 2]{brouwer1989distance}.

\begin{defi}
  Let $X$ be a finite set. A $d$-class \textit{association scheme} is a pair $(X, \scrR)$, where $\scrR = \{ R_0, \ldots, R_d \}$ is a set of non-empty symmetric binary relations on $X$ with the following properties:
  \begin{enumerate}[(a)]
    \item $\scrR$ is a partition of $X \times X$.
    \item $R_{0}$ is the identity relation.
    \item There are integers $p_{ij}^k$ such that for $x, y \in X$ with $x R_k y$ there are exactly $p_{ij}^k$ elements $z$ with $x R_i z$ and $z R_j y$.
  \end{enumerate}
\end{defi}

The number $n_i := p_{ii}^{0}$ is called the \textit{$i$-valency} of $(X, \scrR)$. The total number of elements of $X$ is
\begin{align*}
  n := |X| = \sum_{i=0}^d n_i.
\end{align*}

The relations $R_i$ are described by their adjacency matrices $A_i \in \bbC^{n,n}$ defined by
\begin{align*}
  (A_i)_{xy} = \begin{cases}
                 1 & \text{ if } x R_i y,\\
                 0 & \text{ otherwise.}
               \end{cases}
\end{align*}
The matrices $A_i$ have exactly $d+1$ common eigenspaces $V_j$ with associated eigenvalues $P_{ij}$ (see \cite[p. 45]{brouwer1989distance}).
There exist idempotent Hermitian matrices $E_j \in \bbC^{n,n}$ (hence they are positive semidefinite) with the properties
\begin{align*}
\begin{array}{lll}\displaystyle
\sum_{j=0}^d E_j = I, & \hspace*{2cm} &\displaystyle E_{0} = \frac{1}{n} J,\\
\displaystyle A_j = \sum_{i=0}^d P_{ij} E_i, &  &\displaystyle E_j = \frac{1}{n} \sum_{i=0}^d Q_{ij} A_i,
  \end{array}
\end{align*}
where $J$ is the all-one matrix, and $P = (P_{ij}) \in \bbC^{d+1,d+1}$ and $Q = (Q_{ij}) \in \bbC^{d+1,d+1}$ are the so-called eigenmatrices of the association scheme.

In this paper, $X$ will be the set of generators of a polar space of rank $d$.
The relations $R_0, \ldots R_d$ are defined by
\begin{align*}
  (A_i)_{xy} = \begin{cases}
          1 & \text{ if } \codim(x \cap y) = i,\\
          0 & \text{ if } \codim(x \cap y) \neq i.
         \end{cases}
\end{align*}
for generators $x$, $y$, and $0\leq i \leq d$.

Formulas for the eigenvalues of these association schemes from polar spaces
were calculated by Stanton \cite{MR618532}, Eisfeld \cite{MR1714376}, and Vanhove \cite[Theorem 4.3.6]{vanhove_phd}. We will use Vanhove's version.

\begin{satz}\label{thm_ev_gen}
  The eigenvalues of the adjacency matrix $A_s$ are
  \begin{align*}
   P_{r,s} &= \sum_{t=\max(r-s, 0)}^{\min(d-s,r)} (-1)^{r-t} \gauss{d-r}{d-s-t} \gauss{r}{t} q^{\binom{r-t}{2} + \binom{s-r+t}{2} + (s-r+t) \epsilon}.
  \end{align*}
\end{satz}

By Lemma 2.2.1 (ii) of \cite{brouwer1989distance}, we see that $P_{0s}$ is the number of
generators which meet a fixed generator in codimension $s$.
Hence, the previous formula yields the following well-known result.

\begin{kor}\label{kor_dis_gens}
  In a polar space of type $\epsilon$ and rank $d$, exactly
  \begin{align*}
    \gauss{d}{d-s} q^{\binom{s}{2} + s \epsilon}
  \end{align*}
  generators meet a fixed generator in codimension $s$.
\end{kor}

We are interested in the eigenvalues of $\sum_{s=0}^a A_{d-s}$, so we shall have to explicitly calculate these.
Since the $V_r$ are the common eigenspaces of $A_0, \ldots, A_d$, the $V_r$ are subspaces of eigenspaces of $\sum_{s=0}^a A_{d-s}$, and the eigenvalue of $\sum_{s=0}^a A_{d-s}$ on $V_r$ is
\[
\lambda_r^a := \sum_{s=0}^a P_{r,d-s}.
\]

\begin{satz}\label{thm_formula_ev}\label{lem_formula_ev}
  For $a < d$ we have
  \begin{align*}
    \lambda_{0}^a = \sum_{s=0}^{a} \gauss{d}{s} q^{\binom{d-s}{2} + (d-s)\epsilon}.
  \end{align*}
  For $a < d$ and $r > 0$ we have
  \begin{align*}
    \lambda_r^{a}  = (-1)^{r+a} \sum_{s=\max( a-r+1, 0)}^{\min(a, d-r)} (-1)^s A(r, s, a)
  \end{align*}
  where
  \begin{align*}
    A(r, s, a) := \gauss{d-r}{s} q^{\binom{d-r-s}{2} + (d-r-s)\epsilon} \gauss{r-1}{a-s} q^{\binom{r-a+s}{2}}.
  \end{align*}
\end{satz}
\begin{proof}
  The formula for $\lambda_0^a$ follows from Theorem \ref{thm_ev_gen}. For $\lambda_r^a$ with $r > 0$ we start with Theorem \ref{thm_ev_gen} to see
  \begin{align*}
    \lambda_r^{a} &= \sum_{s=0}^a \sum_{t=\max(r+s-d, 0)}^{\min(s,r)} (-1)^{r-t} \gauss{d-r}{s-t} \gauss{r}{t} q^{\binom{r-t}{2} + \binom{d-r+t-s}{2} + (d-r+t-s)\epsilon}\\
    &= \sum_{t=0}^{r} (-1)^{r-t} \gauss{r}{t} q^{\binom{r-t}{2}} \sum_{s=t}^{\min( a, d-r+t)} \gauss{d-r}{s-t} q^{\binom{d-r+t-s}{2} + (d-r+t-s)\epsilon}\\
    &= \sum_{t=0}^{r} (-1)^{r-t} \gauss{r}{t} q^{\binom{r-t}{2}} \sum_{s=0}^{\min( a-t, d-r)} \gauss{d-r}{s} q^{\binom{d-r-s}{2} + (d-r-s)\epsilon}\\
    &= \sum_{s=0}^{\min( a, d-r)} \gauss{d-r}{s} q^{\binom{d-r-s}{2} + (d-r-s)\epsilon} \sum_{t=0}^{\min(a-s,r)} (-1)^{r-t} \gauss{r}{t} q^{\binom{r-t}{2}}\\
    &\stackrel{\ref{identity_0}}{=} \sum_{s=\max(a-r, 0)}^{\min( a, d-r)} \gauss{d-r}{s} q^{\binom{d-r-s}{2} + (d-r-s)\epsilon} \sum_{t=0}^{a-s} (-1)^{r-t} \gauss{r}{t} q^{\binom{r-t}{2}}\\
    &\stackrel{\ref{identity_0}}{=} \sum_{s=\max(a-r+1, 0)}^{\min\{ a, d-r\}} \gauss{d-r}{s} q^{\binom{d-r-s}{2} + (d-r-s)\epsilon} (-1)^{r+a-s} \gauss{r-1}{a-s} q^{\binom{r-a+s}{2}}\\
    &= (-1)^{r+a} \sum_{s=\max(a-r+1, 0)}^{\min( a, d-r)} (-1)^s \gauss{d-r}{s} q^{\binom{d-r-s}{2} + (d-r-s)\epsilon} \gauss{r-1}{a-s} q^{\binom{r-a+s}{2}}\\
    &= (-1)^{r+a} \sum_{s=\max(a-r+1, 0)}^{\min( a, d-r)} (-1)^s A(r, s, a).
  \end{align*}
\end{proof}

\begin{kor}\label{kor_formula_spec}
  For $a \leq d-1$ we have
  \begin{enumerate}[(a)]
   \item
   \begin{align*}
   \lambda_1^a &= -\gauss{d-1}{a} q^{\binom{d-a-1}{2} + (d-a-1)\epsilon} = -A(1, a, a),
  \end{align*}
   \item
  \begin{align*}
    \lambda_d^a &= (-1)^{d-a} \gauss{d-1}{a} q^{\binom{d-a}{2}} = (-1)^{d+a} A(d, 0, a),
  \end{align*}
   \item
   \begin{align*}
    \lambda_{d-1}^a &= (-1)^{d-1+a}\gauss{d-2}{a} q^{\binom{d-a-1}{2} + \epsilon} +(-1)^{d+a} \gauss{d-2}{a-1} q^{\binom{d-a}{2}}.
   \end{align*}
  \end{enumerate}
\end{kor}

\section{Hoffman's Bound}\label{sec_hoffman}

The famous bound by Hoffman on independent sets restricts the maximum size $c_{d,t}$ of a $(d, t)$-EKR set.
It is known that this bound is sharp for $t=d-1$ except when $\epsilon = \frac{1}{2}$ and $d$ is odd \cite{MR2755082}.

\begin{prop}[Hoffman's Bound {\cite[Proposition 3.7.2]{brouwer1989distance}}]\label{hoffman_bound}
  ~
  \begin{align*}
   c_{d,t} \leq \dfrac{n \lambda_{\min}}{\lambda_{\min} - k},
  \end{align*}
  where $\lambda_{\min} := \min_r \lambda_r^{d-t-1}$ is the smallest eigenvalue of the matrix $\sum_{s=t+1}^d A_{s}$, and $k = \lambda_{0}^{d-t-1}$ is the valency of the graph associated to this matrix.
\end{prop}

To our knowledge the smallest eigenvalue of $\sum_{s=t+1}^d A_{d-s}$ was never calculated except for special cases such as $t = d-1$, so this section is concerned about approximating $\lambda_{\min}$.
Our claim is the following:

\begin{satz}\label{thm_smallest_ev}
  For $a \leq d-1$, and $q \geq 3$, the following holds:
  \begin{enumerate}[(a)]
   \item $|\lambda_1^a|=\max\{ |\lambda_r^a|: r = 1, \ldots, d \}$ if $\epsilon \geq 1$.
   \item $|\lambda_d^a|=\max\{ |\lambda_r^a|: r = 1, \ldots, d \}$ if $\epsilon \leq 1$.
   \item $\lambda_1^a=\min\{ \lambda_r^a: r = 1, \ldots, d \}$ if $d - a$ even or $\epsilon \geq 1$.
   \item $\lambda_d^a=\min\{ \lambda_r^a: r = 1, \ldots, d \}$ if $d - a$ odd and $\epsilon \leq 1$.
  \end{enumerate}
\end{satz}

We shall prove Theorem \ref{thm_smallest_ev} in several steps.

\begin{lemma}\label{lem_thm_smallest_ev_a_eq_dm1}
  Theorem \ref{thm_smallest_ev} holds for $a = d-1$ and for $d\le 2$.
\end{lemma}
\begin{proof}
  Theorem \ref{thm_formula_ev} shows that $\lambda_{r}^{d-1} = -1$ for all $r \in \{1, \ldots, d\}$, so Theorem \ref{thm_smallest_ev} holds for $a=d-1$, and in particular for $(d, a) \in \{ (2, 1), (1, 0) \}$. For $(d, a) = (2, 0)$, Corollary \ref{kor_formula_spec} gives $\lambda_1^0 = -q^e$ and $\lambda_2^0 = q$ and again the assertion follows.
\end{proof}

\begin{prop}\label{prop_simple_comp}
  Let $a \leq d-2$, $q \geq 3$. Then
  \begin{enumerate}[(a)]
   \item \label{prop_simple_comp_a}
  \begin{align*}
    |\lambda_1^a| - |\lambda_d^a|
    \begin{cases}
      > 0 & \text{ if } \epsilon > 1,\\
      = 0 & \text{ if } \epsilon = 1,\\
      < 0 & \text{ if } \epsilon < 1.
    \end{cases}
  \end{align*}
  \item If $d \geq 3$, then \label{prop_simple_comp_b}
    \begin{align*}
     |\lambda_1^a|, |\lambda_d^a| \geq |\lambda_{d-1}^a|
    \end{align*}
  \end{enumerate}
\end{prop}
\begin{proof}
  By Corollary \ref{kor_formula_spec},
  \begin{align*}
    |\lambda_1^a| - |\lambda_d^a| &=  \gauss{d-1}{a} q^{\binom{d-a-1}{2} + (d-a-1)\epsilon} - \gauss{d-1}{a} q^{\binom{d-a}{2}}\\
    &= \gauss{d-1}{a} q^{\binom{d-a}{2}} \left( q^{(d-a-1)(\epsilon-1)} - 1 \right).
  \end{align*}
  As $a \leq d-2$, the statement in \eqref{prop_simple_comp_a} follows.
  For \eqref{prop_simple_comp_b} we calculate
  \begin{align*}
    |\lambda_{d-1}^a| &= \left| \gauss{d-2}{a} q^{\binom{d-a-1}{2}+\epsilon} - \gauss{d-2}{a-1} q^{\binom{d-a}{2}} \right|\\
    &\stackrel{\eqref{gauss_rec}}{=} \left| q^{\binom{d-a-1}{2}+\epsilon} \gauss{d-2}{a} + q^{\binom{d-a-1}{2}} \left( \gauss{d-2}{a} - \gauss{d-1}{a} \right) \right|\\
    &= q^{\binom{d-a-1}{2}} \left| (q^{\epsilon} + 1) \gauss{d-2}{a} - \gauss{d-1}{a} \right|\\
    &= q^{\binom{d-a-1}{2}} \gauss{d-1}{a} \left| (q^{\epsilon} + 1) \frac{q^{d-1-a}-1}{q^{d-1} -1} -1 \right|.\\
    \intertext{If $a < d-2$, then by Corollary \ref{kor_formula_spec}}
    |\lambda_{d-1}^a| &\leq q^{\binom{d-a-1}{2}+e} \gauss{d-1}{a} \leq |\lambda_1^a|, |\lambda_d^a|.
    \intertext{If $a = d-2$, then by $d \geq 3$ and Corollary \ref{kor_formula_spec}}
    |\lambda_{d-1}^a| &\leq q^{\binom{d-a-1}{2}} \gauss{d-1}{a} \left|(q^{2} + 1) \frac{q-1}{q^{d-1} -1} -1\right|\\
    &\leq q^{\binom{d-a-1}{2}+1} \gauss{d-1}{a} \leq |\lambda_1^a|, |\lambda_d^a|.
  \end{align*}
  This shows \eqref{prop_simple_comp_b}.
\end{proof}

\begin{figure}[ht]
 \centering
\begin{tikzpicture}[scale=0.9]
  \draw[->, ultra thick] (0,0) -- (9,0) node[right] {$s$};
  \draw[->, ultra thick] (0,0) -- (0,3) node[above] {$A(r, s, a)$};
  \draw[scale=1,domain=0:8,smooth,variable=\x,black,ultra thick] plot ({\x},{(-\x*\x + 8*\x)/5});
  \node at (3.9,3.7) {$A(r, \frac{a-\epsilon}{2},a)$};
  \draw[fill=circle1,thick] (4,3.2) circle [radius=0.14];
\end{tikzpicture}
 \caption{The function $A(r, s, a)$ imagined as a continuous unimodal function in $s$.}
 \label{fig:Arsa_unimodal}
\end{figure}
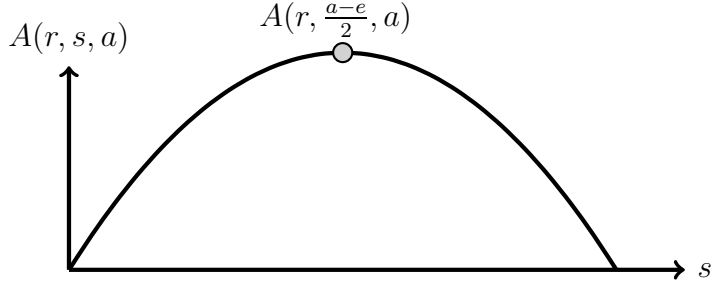

\begin{lemma}\label{lem_monotony_ev}
  For fixed $r > 0$, $q \geq 3$, and $a \in \{ 0, \ldots, d-1\}$ the sequence
  \begin{align*}
    (A(r, s, a))_{\max(a-r+1, 0) \leq s \leq \min(a, d-r)}
  \end{align*}
  is unimodal.
  More precisely, we have
  \begin{enumerate}[(a)]
   \item If $2s + \epsilon - a \geq \frac{1}{2}$, then $A(r, s, a) > A(r, s+1, a)$.
   \item If $2s + \epsilon - a \leq -\frac{1}{2}$, then $A(r, s, a) < A(r, s+1, a)$.
  \end{enumerate}
\end{lemma}
\begin{proof}
  We investigate the sign of $x := A(r, s, a) - A(r, s+1, a)$ for integers $s$ with $\max(a-r+1, 0) \leq s \leq \min(a, d-r)-1$.
  \begin{align*}
    x &= \gauss{d-r}{s} \gauss{r-1}{a-s} q^{\binom{d-r-s}{2} + \binom{r-a+s}{2} + (d-r-s)\epsilon} \\
    &- \gauss{d-r}{s+1} \gauss{r-1}{a-s-1} q^{\binom{d-r-s-1}{2} + \binom{r-a+s+1}{2} + (d-r-s-1)\epsilon}.
  \end{align*}
Using
\[
\gauss{r-1}{a-s}=\gauss{r-1}{a-s-1}\cdot\frac{q^{r-a+s}-1}{q^{a-s}-1}
\ \mbox{and}\
\gauss{d-r}{s+1}=\gauss{d-r-1}{s}\cdot\frac{q^{d-r}-1}{q^{s+1}-1}
\]
we find  that $x= q^y\gauss{r-1}{a-s-1}\gauss{d-r-1}{s} B$ for some integer $y$ and
  \begin{align*}
    B := q^{2s+e-a} \frac{q^{r-a+s}-1}{q^{r-a+s}-q^{2s-2a+r}}-
    \frac{q^{d-r-s}-1}{q^{d-r-s}-q^{d-2s-r-1}}.
  \end{align*}
  As $1\le q^n/(q^n-q^k)\le q/(q-1)$ for integers $0\le k\le n-1$, this implies that
  \[
  q^{2s+e-a}-\frac{q}{q-1}\le B\le q^{2s+e-a}\frac{q}{q-1}-1.
  \]
  Using $q\ge 3$, we find that $B>0$ if $2s+e-a\ge \frac12$, and $B<0$, if $2s+e-a\le -\frac12$.
\end{proof}

\begin{kor}\label{kor_ev_upper_bnd}
  For $q \geq 3$, $a < d$, and $r > 0$ we have
   \begin{align*}
    |\lambda_r^a | \leq \max\{A(r, s, a) : s\in\bbZ,\ \max(a-r+1, 0) \leq s \leq \min(a, d-r)\}.
  \end{align*}
\end{kor}
\begin{proof}
  This is a direct consequence of the formula for $\lambda_r^a$ given in Theorem \ref{thm_ev_gen} and the unimodality of $A(r, s, a)$ given in Lemma \ref{lem_monotony_ev}.
\end{proof}

 \begin{proof}[Proof of \ref{thm_smallest_ev}]
The theorem was proven in Lemma \ref{lem_thm_smallest_ev_a_eq_dm1} for $a=d-1$,
and in \cite[p. 1295]{MR2755082} for $a=0$. Hence, we assume now that $1 \leq a\le d-2$. In view of Proposition  \ref{prop_simple_comp} and Corollary \ref{kor_formula_spec}, it suffices to show that $|\lambda_1^a|\ge |\lambda_r^a|$ for $2\le r\le d-2$. Using $\lambda_1^a=-A(1,a,a)$, Corollary \ref{kor_ev_upper_bnd} shows that it is sufficient to show that $A(1,a,a)\ge A(r,s,a)$ for all integers $s$ satisfying $\max(a-r+1, 0) \leq s \leq \min(a, d-r) \}$. Define
  \begin{align*}
     f(r,s)&:=\binom{d-r-s}{2} + (d-r-s)\epsilon + \binom{r+s-a}{2}
     \\
     &+ (d-r-s)s+(r+s-1-a)(a-s).
  \end{align*}
  for all integers $r,s$ with $a+1\le r+s\le d$. Using the definition of $A(r,s,a)$ in Lemma \ref{thm_formula_ev}, Lemma \ref{lem_lowest_gauss} and Lemma \ref{lem_upper_bnd_gauss} give
  \begin{align*}
   A(1,a,a) &\geq \frac{4}{3}q^{f(1,a)},
   \\
   A(r,s,a) &\leq 4q^{f(r,s)}.
  \end{align*}
  As $q\geq 3$,
  it suffices therefore to show that $f(1,a)\ge f(r,s)+1$ for all $r,s$ with
  $2\leq r\leq d-2$ and $a+1\leq r+s\leq d$ and $0\leq s\leq a$. Consider such a pair $(r,s)$.
  An easy calculation gives \begin{align*}
  f(1,a)-f(r,s)&=(d-a-1)(r-1)-(r+s-a-1)(r-e-s)
  \\
  &\geq (d-a-1)(r-1)-(r+s-a-1)(r-s).
  \end{align*}
  Denote the right hand side by $g(r,s)$. If $s\geq 2$, then $s \leq a$, $s+r\leq d$, $r \geq 2$ and $a \leq d-1$ show that
    \begin{align*}
  g(r,s)  &\geq (d-a-1)(r-1)-(r+s-a-1)(r-2)
  \\&\geq (d-a-1)(r-1)-(d-a-1)(r-2)\\
  &\geq d-a-1 \geq 1.
  \end{align*}
  If $s\in\{0,1\}$, then
  \begin{align*}
    g(r,s)  &= (r-1)(d-r-1)+(1-s)a\geq r-1 \geq 1.
  \end{align*}
  since $r\leq d-2$ and $a\ge 0$.
\end{proof}

We have calculated the smallest eigenvalues and, therefore, Hoffman's bound can be applied. In order to simplify the approximations of the Hoffman's bound in Section \ref{sec_felina}, we provide a simpler formula for the smallest eigenvalue.

\begin{satz}\label{kor_c_d_dmt}\label{thm_hoffman_explicit}
Suppose that $q\ge 3$ and define $\alpha$ by $\alpha \log(1+q^{-\epsilon-1}) = \log(1+ q^{-\epsilon})$. Set $\gamma = 2$ if $q = 3$, and $\gamma = 1 +2q^{-1}$ otherwise.
Let $0 < t < d$.
  \begin{enumerate}[(a)]
   \item We have
   \begin{align*}
    c_{d,t} \leq -\lambda_{\min} (1 + q^{-\epsilon})^{\frac{\alpha}{\alpha-1}},
  \end{align*}
  where $\lambda_{\min} := \min_r \lambda_r^{d-t-1}$ is the smallest eigenvalue of the matrix $\sum_{s=t+1}^d A_{s}$
   \item If $t$ odd or $\epsilon \geq 1$, then
   \begin{align*}
    c_{d,t} \leq \gamma (1 + q^{-\epsilon})^{\frac{\alpha}{\alpha-1}}q^{t(d-t-1) + \binom{t}{2} + t\epsilon}.
  \end{align*}
   \item If $t$ even and $\epsilon \leq 1$, then
   \begin{align*}
    c_{d,t} \leq \gamma (1 + q^{-\epsilon})^{\frac{\alpha}{\alpha-1}}q^{t(d-t-1) + \binom{t+1}{2}}.
  \end{align*}
  \end{enumerate}
\end{satz}
\begin{proof}
  Here we have $a = d-t-1$. An application of Hoffman's Bound, see Proposition \ref{hoffman_bound}, using $\lambda_{\min}\le 0$ shows that  $c_{d,t} \leq -n \lambda_{\min}/k$. Theorem \ref{thm_formula_ev} shows that
  \begin{align*}
    k = \lambda_0^{d-t-1} \geq q^{\binom{d}{2} + d \epsilon}.
  \end{align*}
  Lemma \ref{lem_est_gens} shows that
  \begin{align*}
    n\le q^{\binom{d}{2} + d \epsilon}\cdot (1 + q^{-\epsilon})^{\frac{\alpha}{\alpha-1}}.
  \end{align*}
  Using these estimations, we find the bound for $c_{d,t}$ given in (a). From Theorem \ref{thm_smallest_ev} and Lemma \ref{kor_formula_spec}. Using Lemma \ref{lem_upper_bnd_gauss}, we find
  \begin{align*}
    -\lambda_{\min} &= \gauss{d-1}{t} q^{\binom{t}{2} + t\epsilon}
    \leq \gamma q^{t(d-t-1) + \binom{t}{2} + t\epsilon}
  \end{align*}
  if $t$ odd or $\epsilon \geq 1$, and
  \begin{align*}
    -\lambda_{\min} &= (-1)^{t} \gauss{d-1}{t} q^{\binom{t+1}{2}}
    \leq \gamma q^{t(d-t-1) + \binom{t+1}{2}}
  \end{align*}
  if $t$ even and $\epsilon \leq 1$. Now (b) and (c) follow from (a).
\end{proof}

\section{Proof of the Main Theorem}\label{sec_felina}

In this section we want to specify the $q$, $d$, and $t$ for which our results in the two sections are non-trivial statements.
We shall do so by providing lower, respectively, upper bounds on all the parameters
used in Theorem \ref{thm_max_t_even} and Theorem \ref{thm_max_t_odd}.
For this we shall provide some upper estimates for $b_1^\even, b_2^\even, b_1^\odd, b_2^\odd, b_3^\odd$.
Throughout this section, $q$ is fixed and we define $\alpha$ and $\gamma$ as follows.
\begin{enumerate}[(i)]
 \item $\gamma := 2$ if $q=3$, and $\gamma := 1+2q^{-1}$
  if $q \geq 4$.
 \item $\alpha$ is chosen as in Lemma \ref{lem_est_gens}, that is $\alpha\cdot \log(1+q^{-\epsilon-1}) = \log(1+q^{-\epsilon})$.
\end{enumerate}

\begin{lemma}\label{lem_psi_bounds}
\begin{enumerate}[(a)]
   \item If $t$ is even and $5t \leq 2d+1$, then
    \begin{align*}
     \psi^\even \leq q^{\frac{3}{4} t^2 + \frac{t}{2}(d-2t) -(d-\frac{5}{2} t + 2)} \tfrac{\gamma^2}{1-q^{-2}}.
    \end{align*}
    \item If $t$ is odd and $5t \leq 2d+1$, then
    \begin{align*}
     \psi^\odd \leq q^{(\frac{t}{2} - \frac{3}{2})(\frac{3}{2}t - \frac{1}{2}) + (d-2t+1)(\frac{t}{2}-\frac{1}{2}) -(d-\frac{5}{2} t + \frac{7}{2})} \tfrac{\gamma^2}{1-q^{-2}}.
    \end{align*}
\end{enumerate}
\end{lemma}
\begin{proof}
(a)  For integers $i$ with $1 \leq i \leq \frac{t}{2} - 1$, Lemma \ref{lem_upper_bnd_gauss} shows that
  \begin{align*}
    &q^{(\frac{t}{2}-1-i)(\frac{t}{2}-i)} \gauss{d-2t+1}{t/2 - i} \gauss{t/2-1}{i}\\
    &\leq  \gamma^2 q^{(\frac{t}{2}-1-i)(\frac{t}{2}-i) + (\frac{t}{2}-i)(d-\frac{5}{2}t + 1 + i) + i(\frac{t}{2}-1-i)}\\
    &=  \gamma^2 q^{\frac{t}{2}(d-2t) - i(d-\frac{5}{2} t + 1 + i)}.
  \end{align*}
  Using $5t\le 2d$ and Lemma \ref{lem_basic_cnt_pgnq}, we find that
  \begin{align*}
    \psi^\even &= q^{\frac{3}{4} t^2} \sum_{i=1}^{t/2-1} q^{(\frac{t}{2}-1-i)(\frac{t}{2}-i)} \gauss{d-2t+1}{t/2 - i} \gauss{t/2-1}{i}\\
    &\leq q^{\frac{3}{4} t^2 + \frac{t}{2}(d-2t)} \gamma^2 \sum_{i=1}^{t/2-1} q^{-i(d-\frac{5}{2} t + 1 + i)}\\
    &\leq q^{\frac{3}{4} t^2 + \frac{t}{2}(d-2t) -(d-\frac{5}{2} t + 2)} \gamma^2 \sum_{i=0}^{t/2-2} q^{-2i}\\
    &\leq q^{\frac{3}{4} t^2 + \frac{t}{2}(d-2t) -(d-\frac{5}{2} t + 2)} \tfrac{\gamma^2}{1-q^{-2}}.
  \end{align*}
(b) For integers $i$ with $1 \leq i \leq \frac{t-3}{2}$, Lemma \ref{lem_upper_bnd_gauss} shows that
  \begin{align*}
    &q^{(\frac{t}{2}-\frac{3}{2}-i)(\frac{t}{2}-\frac{1}{2}-i)}
      \gauss{d-2t+2}{(t-1)/2-i} \gauss{t/2 - \frac{3}{2}}{i}\\
    &\leq q^{(\frac{t}{2}-\frac{3}{2}-i)(\frac{t}{2}-\frac{1}{2}-i) + (\frac{t}{2}-\frac{1}{2}-i)(d-\frac{5}{2}t + \frac{5}{2} + i) + i(\frac{t}{2}-\frac{3}{2}-i)} \gamma^2 \\
    &= q^{(d-2t+1)(\frac{t}{2}-\frac{1}{2}) - i (d-\frac{5}{2} t + \frac{5}{2} + i)} \gamma^2.
  \end{align*}
  Using $5t\le 2d+1$ and Lemma \ref{lem_basic_cnt_pgnq}, we find that
  \begin{align*}
    \psi^\odd &= q^{(\frac{t}{2}-\frac{1}{2})(\frac{3}{2}t-\frac{3}{2})} \sum_{i=1}^{t/2-3/2} q^{(\frac{t}{2}-\frac{3}{2}-i)(\frac{t}{2}-\frac{1}{2}-i)}
      \gauss{d-2t+2}{(t-1)/2-i} \gauss{\frac{t}{2} - \frac{3}{2}}{i}\\
    &\leq q^{(\frac{t}{2}-\frac{1}{2})(\frac{3}{2}t-\frac{3}{2}) + (d-2t+1)(\frac{t}{2}-\frac{1}{2})} \gamma^2 \sum_{i=1}^{t/2-3/2} q^{-i (d-\frac{5}{2} t + \frac{5}{2} + i)}\\
    &\leq q^{(\frac{t}{2}-\frac{1}{2})(\frac{3}{2}t-\frac{3}{2}) + (d-2t+1)(\frac{t}{2}-\frac{1}{2}) - (d-\frac{5}{2} t + \frac{7}{2})} \gamma^2 \sum_{i=0}^{t/2-3/2} q^{-2i}\\
    &\leq q^{(\frac{t}{2}-\frac{1}{2})(\frac{3}{2}t-\frac{3}{2}) + (d-2t+1)(\frac{t}{2}-\frac{1}{2}) -(d-\frac{5}{2} t + \frac{7}{2})}  \tfrac{\gamma^2}{1-q^{-2}}.
  \end{align*}
  This shows part (b).
\end{proof}

\begin{lemma}\label{lem_appx_bnd}
Suppose that $t\ge 2$, and $5t \leq 2d$, and $q\ge 3$. Then
  \begin{align*}
    b_1^\even &\leq q^{(\frac{t}{2} - 1)(d-2t+1) + t(t-2) + \binom{t}{2} + t\epsilon} \gamma^2 (1+q^{-\epsilon})^{\frac{\alpha}{\alpha-1}} \text{ if } \epsilon \geq 1,
    \\
    b_1^\even &\leq q^{(\frac{t}{2} - 1)(d-2t+1) + t(t-2) + \binom{t+1}{2}} \gamma^2 (1+q^{-\epsilon})^{\frac{\alpha}{\alpha-1}} \text{ if } \epsilon \leq 1,
    \\
    b_1^\odd &\leq q^{\frac{t-3}{2} (d-2t+2) + (t-3)t + \binom{t}{2} + t\epsilon} \gamma^2 (1+q^{-\epsilon})^{\frac{\alpha}{\alpha-1}},
    \\
    b_2^\even &\leq q^{\epsilon \frac{t}{2} + \binom{t/2}{2} + \frac{2t(d+5) - t^2 - 4d - 8}{4}} \tfrac{\gamma^2}{1-q^{-2}},
    \\
    b_2^\odd &\leq q^{\epsilon \frac{t+1}{2} + \binom{(t+1)/2}{2} + \frac{2t(d+5) - t^2 - 6d - 13}{4}} \tfrac{\gamma^2}{1-q^{-2}} (1+q^{-\epsilon})^{\frac{\alpha}{\alpha-1}},
    \\
    b_3^\odd &\leq q^{\epsilon \frac{t-1}{2} + \binom{(t-1)/2}{2} + \frac{2t(d+1) - t^2 - 2d - 1}{4}} \gamma.\\
  \end{align*}
\end{lemma}
\begin{proof}
We recall the following definitions.
  \begin{align*}
    b_1^\even &= \gauss{d - \frac{3}{2} t}{t/2-1} c_{ 2t-1, t}\\
    b_1^\odd &= \gauss{d-\frac{3}{2}t+\frac{1}{2}}{(t-3)/2} c_{2t-2, t}\\
    b_2^\even &= \psi^\even q^{\epsilon \frac{t}{2} + \binom{t/2}{2}}\\
    b_2^\odd &= \psi^\odd \omega(d, t/2+\frac{1}{2}) = \psi^\odd \prod_{i=0}^{t/2-\frac{1}{2}} (q^{i+\epsilon} + 1) \\
    b_3^\odd &= \overline{\psi}^\odd q^{\epsilon \frac{t-1}{2} + \binom{(t-1)/2}{2}} =  q^{(\frac{3}{2}t - \frac{1}{2})(\frac{t}{2} - \frac{1}{2})} \gauss{d - \frac{3}{2}t + \frac{1}{2}}{(t-1)/2} q^{\epsilon \frac{t-1}{2} + \binom{(t-1)/2}{2}}
  \end{align*}
 Recall also that the numbers $b^\even_*$ are defined only when $t$ is even and that the numbers $b^\odd_*$ are defined only when $t$ is odd. Using these definitions, the bounds for $b_1^0$ and $b_1^1$ follow from Theorem \ref{thm_hoffman_explicit} and Lemma \ref{lem_upper_bnd_gauss}, the bound for $b_2^0$ follows from Lemma \ref{lem_psi_bounds}, the bound for $b_2^1$ follows from Lemma \ref{lem_psi_bounds} and Lemma \ref{lem_est_gens}, and the bound for $b^1_3$ follows from Lemma \ref{lem_upper_bnd_gauss}. \end{proof}

Additionally, we need lower bounds for the size of our examples.
Recall that the case $t = d-1$ is not covered by Theorem \ref{thm_apprx} and that the case $t = 1$ is trivial (For $t=1$ the only maximal EKR-sets under investigation consist of all generators on a totally isotropic subspace of rank $d-1$). In order to prove Theorem \ref{thm_apprx}, we can therefore assume that $2 \leq t \leq d-2$.

\begin{lemma} Suppose that $d - 2 \geq t \geq 2$. \label{apprx_ex}
Then the EKR-sets described in Example \ref{lem_ex_t_even} have size at least
\begin{align*}
  y^\even := \gauss{d}{\frac{t}{2}} q^{\epsilon \frac{t}{2} + \binom{t/2}{2}} \geq q^{\epsilon \frac{t}{2} + \binom{t/2}{2} + \frac{t}{2}(d - \frac{t}{2})} (1+q^{-1}).
\end{align*}
and the ones described in Example \ref{lem_ex_t_odd} have size at least
\begin{align*}
  y^\odd := \gauss{d-1}{\frac{t-1}{2}} q^{\epsilon \frac{t+1}{2} + \binom{\frac{t+1}{2}}{2}} ( 1 + q^{-\epsilon} ) \geq q^{\epsilon \frac{t+1}{2} + \binom{(t+1)/2}{2} + \frac{t-1}{2} (d - \frac{t+1}{2})} (1+q^{-1}) ( 1 + q^{-\epsilon} ).
\end{align*}
\end{lemma}
\begin{proof}
  \textbf{Case $t$ even.}
  Let $G_0$ be a generator and let $Y'$ be the set of all generators $G$ with $\dim(G\cap G_0)=d-\frac{t}{2}$. Obviously, $Y'$ has less elements than the EKR-sets constructed in Example \ref{lem_ex_t_even}. Lemma \ref{kor_dis_gens} shows $|Y'|=y^\even$.
  Lemma \ref{lem_lowest_gauss} proves the lower bound for $y^\even$.

  \textbf{Case $t$ odd.}
  Let $G$ be a generator and $U$ a subspace of $G$ of codimension one.   Let $Y'$ be the set of all generators which meet $U$ in dimension $d - \frac{t}{2} - \frac{1}{2}$.
  Obviously, $Y'$ has at most as many elements as the EKR-sets constructed in Example \ref{lem_ex_t_odd}. We shall show $|Y'|=y^\odd$. Let $T$ be one of the $\gauss{d-1}{\frac{t-1}{2}}$ subspaces of $U$ of dimension $d - \frac{t}{2} - \frac{1}{2}$.

  If $H \in Y'$ with $U \cap H = T$, then $G \cap H = T$ or $V := G \cap H$ has dimension $\dim(T)+1$ and satisfies $V \cap U = T$. In the quotient geometry on $T$, Corollary \ref{kor_dis_gens}   shows that there exist
  \begin{align*}
    q^{\epsilon \frac{t+1}{2} + \binom{\frac{t+1}{2}}{2}}
  \end{align*}
  generators $H$ with $H \cap G = T$. The number of subspaces $V$ of $G$ with
  $\dim(V) = \dim(T)+1$ and $V \cap U = T$ is $q^{\frac{t-1}{2}}$
  as can be seen in the quotient geometry on $T$. For each such $V$, Corollary \ref{kor_dis_gens}
  applied to the quotient geometry of $V$ shows that there are
  \begin{align*}
    q^{\epsilon \frac{t-1}{2} + \binom{\frac{t-1}{2}}{2}}
  \end{align*}
  generators $H$ with $H \cap G = V$. Hence
  \begin{align*}
    |Y'|=\gauss{d-1}{\frac{t-1}{2}} \left( q^{\epsilon \frac{t+1}{2} + \binom{\frac{t+1}{2}}{2}} +  q^{\frac{t-1}{2}} \cdot q^{\epsilon \frac{t-1}{2} + \binom{\frac{t-1}{2}}{2}}\right)
    = \gauss{d-1}{\frac{t-1}{2}} q^{\epsilon \frac{t+1}{2} + \binom{\frac{t+1}{2}}{2}} \left( 1 + q^{-\epsilon} \right).
  \end{align*}
  Lemma \ref{lem_lowest_gauss} shows the remaining inequality.
  \end{proof}

All left to do is to compare $b_1^\even + b_2^\even$, respectively, $2b_1^\odd + b_2^\odd + b_3^\odd$ to the sizes of the examples ($y^\even$, respectively, $y^\odd$) using the given upper, respectively, lower bounds.
Then Theorem \ref{thm_max_t_even} and Theorem \ref{thm_max_t_odd} yield our last theorem.
Hence, we compare all degrees of the bounds in $q$ to $y^\even$, respectively, $y^\odd$.
This yields for $q \geq 3$,
\begin{align*}
  &\delta_1^\even := \deg(y^\even) - \deg(b_1^\even) = \begin{cases} d + 1 - \frac{2\epsilon+1}{4} t - \frac{5}{8} t^2 \text{ if } \epsilon \geq 1,\\
  d + 1 - \frac{5-2\epsilon}{4} t - \frac{5}{8} t^2 \text{ if } \epsilon \leq 1,\end{cases}\\
  &\delta_2^\even := \deg(y^\even) - \deg(b_2^\even) = d + 2 - \frac{5}{2} t,\\
  &\delta_1^\odd := \deg(y^\odd) - \deg(b_1^\odd) = d + \frac{25}{8} + \epsilon/2 - \frac{(\epsilon+1)t}{2} - \frac{5}{8} t^2,\\
  &\delta_2^\odd := \deg(y^\odd) - \deg(b_2^\odd) = d + \frac{7}{2} - \frac{5}{2} t,\\
  &\delta_3^\odd := \deg(y^\odd) - \deg(b_3^\odd) = \epsilon.
\end{align*}

These approximations make it clear that Theorem \ref{thm_max_t_even} and Theorem \ref{thm_max_t_odd} are non-trivial for $d$ large and $t$ fixed.
In the following we want to be more specific about the necessary size of $d$.
Recall that $\gamma = 2$ if $q=3$, and $\gamma = 1+2q^{-1}$ if $q \geq 4$.

\begin{lemma}\label{lem_spec_tables}
   Let $q$ be an integer with $q\ge 3$.
    \begin{enumerate}[(a)]
   \item $q^z (1+q^{-1}) \geq \gamma^2 ((1+q^{-\epsilon})^{\frac{\alpha}{\alpha-1}} + \tfrac{1}{1-q^{-2}})$ for all integers $z\ge 3$.
   \item $q^z (1+q^{-1})(1+q^{-\epsilon}) \geq \gamma q^{z-\epsilon} + \gamma^2 (1+q^{-\epsilon})^{\frac{\alpha}{\alpha-1}} (2 + \tfrac{1}{1-q^{-2}})$ for all integers $z\ge 4$.
  \end{enumerate}
\end{lemma}
\begin{proof}
  (a) We may assume that $z=3$. In view of Corollary \ref{lem_log_quotient2}, we may also assume that $e=0$. Then the inequality is easily checked when $q=3$ and $\gamma=2$. Suppose now that $q\ge 4$ and $\gamma=1+2q^{-1}$. The left hand side of the inequality is monotonically increasing in $q$. Each term $\gamma^2=(1+2/q)^2$, $1/(1-q^{-2})$ and $(1+q^{-\epsilon})^{\frac{\alpha}{\alpha-1}}$ (with $\epsilon=0$) on the right hand side is monotonically decreasing in $q$, for the last term this follows again from Corollary \ref{lem_log_quotient2}. It therefore remains to verify the inequality for $q=4$ and $\epsilon=0$. This is straightforward

  (b) We write the inequality in the form
  \[
    q^z (1+q^{-1}) -(\gamma-1-q^{-1}) q^{z-\epsilon}\geq  \gamma^2 (1+q^{-\epsilon})^{\frac{\alpha}{\alpha-1}} (2 + \tfrac{1}{1-q^{-2}}).
  \]
  For fixed $q$, Corollary \ref{lem_log_quotient2} shows that the right hand side is monotonically decreasing in $e$, and $\gamma\ge 1+q^{-1}$ implies that the left hand side is monotonically increasing in $e$. Hence it is sufficient to verify the inequality for $\epsilon=0$. As $\gamma\le 2$, it then suffices to verify the inequality for $\epsilon=0$ and $z=4$. This is easily done by hand for $q=3$. For $q\ge 4$, we have $\gamma=1+2q^{-1}$, and thus we have to prove that
  \[
    q^4\geq (1+2q^{-1})^2(2 + \tfrac{1}{1-q^{-2}}) 2^{\frac{\alpha}{\alpha-1}}
  \]
  On the right hand side, the first two factors are obviously monotonically decreasing in $q$ and Lemma \ref{lem_log_quotient2} shows the same for third factor. Since the left hand side is monotonically increasing in $q$, it suffices to verify the inequality for $q=4$, which is easy.
\end{proof}

\begin{proof}[Proof of Theorem \ref{thm_apprx}]
  By hypothesis in Theorem \ref{thm_apprx} we have $8d\ge 5t^2+20t+20$.

  \textbf{Case $t$ is even} For $t=0$, the assertion of Theorem \ref{thm_apprx} is trivial, so we may assume that $t\ge 2$. In view of Theorem \ref{thm_max_t_even} and Lemma \ref{apprx_ex}, we have to show that $y^\even > b_1^\even + b_2^\even$. This follows from the bounds on $b_0^\even$, $b_1^\even$ and $y^\even$ given in Lemmas \ref{lem_appx_bnd} and \ref{apprx_ex} and from the assertion in (a) of Lemma \ref{lem_spec_tables} provided that $\delta_1^\even\ge 3$ and $\delta_2^\even\ge 3$.
  This follows easily from $8d\ge 5t^2+20t+20$ and $t\ge 2$.

  \textbf{Case $t$ is odd} For $t=1$, it is well-known (and easy to see) that the largest $(d, 1)$-EKR set is the set of all generators through a fixed $(d-1)$-space. A proof can be found in \cite[Theorem 6.4.10]{vanhove_phd}. Thus Theorem \ref{thm_apprx} is correct for $t=1$, and we may assume that $t\ge 3$. The argument is now similar to the one used in the case when $t$ is even, but now we use \ref{thm_max_t_odd} and have to show that $y^\odd > 2b_1^\odd + b_2^\odd + b_3^\odd$. As before, this follows from the bounds on $b_i^\odd$ and $y^\odd$ given in Lemmas \ref{lem_appx_bnd} and \ref{apprx_ex} and from the assertion in (b) of Lemma \ref{lem_spec_tables} provided that $\delta_1^\even\ge 4$ and $\delta_2^\even\ge 4$.   This follows easily from $8d\ge 5t^2+20t+20$ and $t\ge 3$.
\end{proof}

\begin{Remark}
  \begin{enumerate}[(a)]
   \item Obviously, even the trivial upper bound for $c_{ 2t-1, t}$, i.e. the number of generators in a polar space of rank $2t-1$, is independent
   of $d$. If one uses this bound instead of Theorem \ref{thm_hoffman_explicit} to bound $c_{2t-1,t}$, then the restriction on $t$ is
   approximately $t \leq \sqrt{\frac{8d}{9}}$ as in the $q=2$ case.
   \item If one uses the linear programming bound instead of Hoffman's bound to approximate $c_{d,t}$, then computer results suggest that the conditions on $t$ in Theorem \ref{thm_apprx} should simplify to approximately $t \leq 2 \sqrt{2d}$ for $d$ large.\footnote{The first author provides a list of the conjectured linear programming bounds on his homepage.}
   \item If one could prove that $c_{2t-1, t}$ is the size of Example \ref{lem_ex_t_even}, respectively, that $c_{ 2t-1, t}$ is the size of Example \ref{lem_ex_t_odd}, then the conditions on $t$ would improve to approximately $t \leq \frac{2}{5} d$.
   So it would be sufficient to focus on these cases to improve the results significantly.
  \end{enumerate}
\end{Remark}

\section{Conclusions}

The authors started their work on this project in the hope that it would be reasonable simple to generalize the classification of $(d, d-1)$-EKR sets of maximum size provided in \cite{MR2755082} by applying Hoffman's bound or one of its generalizations since Hoffman's bound is tight in this case \cite{MR618532} if $\epsilon \neq \frac{1}{2}$.
It turns out that for nearly all $(d, t)$-EKR sets Hoffman's bound is far larger than the largest known examples.

In general, linear programming could be used to obtain better algebraic bounds for all $d$.
While computer results suggest that these upper bounds should be able to improve Theorem \ref{thm_apprx} to approximately $t \leq 2 \sqrt{2d}$, even these bounds are still far away from the largest known examples.
Hence, the authors had to rely explicitly on the geometrical properties of polar spaces for the classification.
It might be very interesting to find a purely algebraical proof of the presented results, since our approach stops working as soon as $t$ is too large compared to $d$, while techniques from algebraic combinatorics seem to work the best when $t$ is large compared to $d$.

In general, a classification of all $(d, t)$-EKR sets seems to be very desirable, since we conjecture that it would turn out to be the following, nice looking result.

\begin{conjecture}\label{conj_classification_d_t_ekr}
  Let $Y$ be a $(d, t)$-EKR set of maximum size. Then one of the following cases occurs:
  \begin{enumerate}[(a)]
   \item $Y$ is the set of all generators on a fixed $(d-t)$-space.
   \item $t$ is even and $Y$ is the set of all generators meeting a fixed generator in at least dimension $d-\frac{t}{2}$.
   \item $t$ is odd and $Y$ is a set of all generators meeting a fixed $(d-1)$-dimensional space in at least dimension $d - \frac{t}{2} - \frac{1}{2}$.
   \item $\epsilon = 0$, $t = d-1$, $d$ is odd and $Y$ is the largest example for $Q^+(2d-1, q)$ as given in \cite{MR2755082}.
  \end{enumerate}
\end{conjecture}

\end{document}